\documentclass[a4paper,reqno]{amsart}

\usepackage{amsmath}
\usepackage{amssymb}
\usepackage{amscd}
\usepackage{amsthm}
\usepackage{type1cm}
\usepackage{tcolorbox}
\usepackage{mathrsfs}

\usepackage[colorlinks,citecolor=darkgreen,linkcolor=red]{hyperref}
\definecolor{darkgreen}{rgb}{0.0, 0.6, 0.0}

\usepackage[all]{xy}
\input xy
\xyoption{all}
\tcbuselibrary{breakable, skins, theorems}

\usepackage{tikz}
\usepackage{tikz-3dplot} 
\usetikzlibrary{math} 
\usetikzlibrary{arrows.meta} 

\def\I{\mathcal{I}}
\def\J{\mathcal{J}}

\DeclareMathOperator{\md}{\mathsf{mod}}
\renewcommand{\mod}{\md}
\DeclareMathOperator{\proj}{\mathsf{proj}}
\DeclareMathOperator{\refl}{\mathsf{ref}}

\DeclareMathOperator{\add}{\mathsf{add}}

\DeclareMathOperator{\Hom}{Hom}
\DeclareMathOperator{\End}{End}

\DeclareMathOperator{\Frac}{Frac}
\DeclareMathOperator{\Ker}{Ker}
\DeclareMathOperator{\Cok}{Cok}
\DeclareMathOperator{\im}{Im}
\renewcommand{\Im}{\im}

\DeclareMathOperator{\Cl}{Cl}
\DeclareMathOperator{\Spec}{Spec}

\DeclareMathOperator{\ch}{char}

\DeclareMathOperator{\rk}{rk}
\DeclareMathOperator{\Conv}{Conv}
\DeclareMathOperator{\vol}{vol}
\DeclareMathOperator{\ord}{ord}

\def\gl{\mathop{\rm gl.dim}\nolimits}

\def\pd{\mathop{\rm proj.dim}\nolimits}

\theoremstyle{definition}
\newtheorem{Thm}{Theorem}[section]
\newtheorem{Lem}[Thm]{Lemma}
\newtheorem{Prop}[Thm]{Proposition}
\newtheorem{Cor}[Thm]{Corollary}
\newtheorem{Def}[Thm]{Definition}
\newtheorem{Ex}[Thm]{Example}
\newtheorem{Rem}[Thm]{Remark}

\newcommand{\FRAC}[2]{\leavevmode\kern.1em\raise.5ex\hbox{\the\scriptfont0 #1}\kern-.1em/\kern-.15em\lower.25ex\hbox{\the\scriptfont0 #2}}

\title[NCCRs of toric singularities with divisor class group of rank one]{Non-commutative crepant resolutions of toric singularities with divisor class group of rank one}
\author{Ryu Tomonaga}
\address{Graduate School of Mathematical Sciences, The University of Tokyo, 3-8-1 Komaba, Meguro-ku, Tokyo, 153-8914, Japan}
\email{ryu-tomonaga@g.ecc.u-tokyo.ac.jp}

\begin{document}
\begin{abstract}
We classify the basic reflexive modules giving toric non-commutative
crepant resolutions (NCCRs) of Gorenstein toric singularities whose
divisor class group has rank one. Their isomorphism classes are
parametrized by nonempty proper upper sets in a partially ordered
quotient of the class group. Each upper set determines a finite
subset called its window, which specifies the divisorial summands.
The number of summands in each such module equals the normalized volume of the associated lattice
polytope, as conjectured by Van den Bergh. If two upper sets differ
by one element, their modules are related by a finite sequence of
Iyama--Wemyss mutations. Such sequences connect any two of the
classified modules, so all toric NCCRs of a fixed singularity are
derived equivalent.

We construct higher-dimensional analogues of dimer dual quivers
when the associated $d$-dimensional lattice polytope is non-pyramidal. The same window
determines a cut of a periodic quiver on the $d$-torus. The cut gives
an explicit presentation of the endomorphism algebra by a quiver
with relations. For $d=2$, the resulting quiver is dual to a dimer
model, a bipartite tiling of the two-torus.
\end{abstract}

\maketitle
\tableofcontents

\section*{Introduction}

Van den Bergh introduced non-commutative crepant resolutions (NCCRs)
as non-commutative analogues of crepant resolutions of Gorenstein
normal singularities \cite{VdB04a}, motivated by the derived McKay
correspondence \cite{BKR}. The theory connects birational geometry
with Cohen--Macaulay representation theory \cite{Iya07b,Yos90},
higher Auslander--Reiten theory \cite{Han24a,Iya07a,Tom25b},
Calabi--Yau algebras \cite{IR}, and additive categorification of
cluster algebras \cite{AIR15,BMRRT,HI22}. NCCRs have been constructed
for classes of quotient singularities \cite{Iya07a,Tom24}, compound
Du Val singularities with crepant resolutions \cite{VdB04b}, and cones
over del Pezzo surfaces, including Du Val cases \cite{Tom25a,VdB04a}.
Further examples, including toric ones, arise from Segre products
\cite{Han25}, minimal nilpotent orbit closures of type $A$ \cite{Har},
and reductive quotients \cite{SVdB17}.

Van den Bergh conjectured that every affine Gorenstein toric variety
admits an NCCR \cite[Conjecture~3.7]{VdB23}. For the coordinate ring $R$
of such a variety, we say that a reflexive $R$-module $M$ gives a
toric NCCR if $M$ is a nonzero direct sum of divisorial modules and
$\Gamma=\End_R(M)$ is maximal Cohen--Macaulay over $R$ and has
finite global dimension. We call $\Gamma$ a toric NCCR of $R$.
Broomhead's dimer construction gives such resolutions in dimension
three over $\mathbb{C}$ \cite[Theorem~8.6]{Bro}, and further
constructions apply to particular families of toric rings
\cite{HN,Mat22,SVdB20a,SVdB20b}. We call a direct sum of divisorial
$R$-modules \emph{basic} if its summands are pairwise non-isomorphic.
Write $G=\Cl(R)$. Since $G$ parametrizes the isomorphism classes of
divisorial modules, the classification amounts to determining which
finite subsets of $G$ give toric NCCRs.

We consider $\rk G=1$, when the defining cone has one more ray
than its dimension. A four-dimensional example of \v{S}penko and Van den Bergh with
$G\cong\mathbb{Z}^2$ admits only non-toric NCCRs
\cite[Example~10.1]{SVdB17}\cite[\S~6]{SVdB20a}.

In rank one, upper sets in a partially ordered quotient of $G$
parametrize the basic modules giving toric NCCRs. The window of each
such set consists of its first element in each orbit of a fixed translation.
The inverse image of the window in $G$ indexes the divisorial
summands. Deleting a minimal element shifts one representative in the
window, and a sequence of Iyama--Wemyss mutations realizes the exchange
of summands.
For a non-pyramidal polytope, the same window determines a cut of a
periodic quiver, from which we construct the arrows presenting the
endomorphism algebra.

The associated lattice polytope $P$ has $d+2$ vertices in dimension
$d=\dim R-1$. Let $k$ be a field and write $R=S_0$, where $S$ is the
$G$-graded polynomial ring over $k$ of Section~\ref{toricGradingData}.
The class group $G$ may have torsion. The coefficients of the
primitive affine relation among the vertices give the weights in
the free quotient of $G$, up to a common sign. Denote the weights
with positive, negative and zero images in
$G/G_{\rm tors}\cong\mathbb{Z}$ by
$\vec{x}_0,\ldots,\vec{x}_l$, $\vec{x}'_0,\ldots,\vec{x}'_{l'}$ and
$(\vec{y}_r)_r$, respectively. Let $q\colon G\to H=G/\langle\vec{y}_r\rangle$
be the quotient map. Order $H$ by $h\leq h'$ if $h'-h$ belongs to the
monoid generated by $q(\vec{x}_i)$ and $-q(\vec{x}'_j)$. The Gorenstein
condition gives
\[
s=\sum_i q(\vec{x}_i)=-\sum_j q(\vec{x}'_j).
\]
An upper set contains every element above any of its elements. An
upper set $I\subseteq H$ is \emph{non-trivial} if
$\varnothing\subsetneq I\subsetneq H$. The \emph{window} of such an $I$
is $J(I)=I\setminus(I+s)$, which contains exactly one representative
of each coset of $\mathbb{Z}s$ in $H$. For a finite $T\subseteq G$, write
$M_T=\bigoplus_{\vec{g}\in T}S_{\vec{g}}$, where $S_{\vec{g}}$ is the
divisorial module of class $\vec{g}$.

\begin{Thm}[Theorem \ref{upNCCRcorr}]
Suppose that $\rk G=1$. The assignment
\[
I\longmapsto [M_{q^{-1}(J(I))}]
\]
is a bijection from the non-trivial upper sets of $H$ to the
isomorphism classes of basic reflexive $R$-modules giving toric
NCCRs of $R$.
\end{Thm}

The windows in Theorem~\ref{upNCCRcorr} are precisely the maximal
Cohen--Macaulay cliques studied by \v{S}penko and Van den Bergh for
generic unimodular representations of a one-dimensional torus over an
algebraically closed field of characteristic zero. Their argument uses
Bell's residue-class count and Koszul complexes at extremal weights to
obtain derived equivalences with the NCCR from a standard interval
\cite[Theorem~1.3, \S~7 and Appendix~A]{SVdBB22}.
Theorem~\ref{toricmut} realizes every covering relation in the
inclusion order on the upper sets of our classification by a sequence
of Iyama--Wemyss mutations. For a non-pyramidal polytope, the same cover
switches cut membership precisely for the arrows entering or leaving
one vertex of the periodic quiver.

Over an algebraically closed field of characteristic zero, the
quasi-symmetric torus construction and induction from the identity
torus of the diagonalizable group with character group $G$ give a
toric NCCR also when $G$ has torsion
\cite[Theorem~1.6.2 and Lemma~4.5.1]{SVdB17}.
For cones over simplicial reflexive $d$-polytopes with at most $d+2$
vertices, Malter and Sheshmani obtain NCCRs from tilting bundles
\cite[Theorem~5.1]{MS25}. The first version of this paper, posted in
October 2025, gave the rank-one upper-set classification. Malter and
Sheshmani subsequently proved existence of toric NCCRs for almost
simplicial Gorenstein cones by descent to faces
\cite[Theorem~4.5]{MS26}.

For $\vec{g}\in G$, Theorem~\ref{divCM} gives the criterion
\[
S_{\vec{g}}\ \text{is maximal Cohen--Macaulay}
\quad\Longleftrightarrow\quad
q(\vec{g})\ngeq s\ \text{and}\ q(\vec{g})\nleq -s.
\]
By Proposition~\ref{refhom},
$\Hom_R(S_{\vec{g}},S_{\vec{h}})\cong S_{\vec{h}-\vec{g}}$, so this
criterion applies to the differences of summand classes. The maximal
subsets of $H$ whose pairwise differences satisfy the criterion are
precisely the windows $J(I)$. We use Koszul complexes attached to
covering relations to prove finite global dimension in
Theorem~\ref{upNCCRcorr} and to construct the mutation approximations
\cite{IW}.

Each covering relation under inclusion has the form
$I\supset I^-=I\setminus\{m\}$ for a minimal element $m$ of $I$.
The new window is
\[
J(I^-)=(J(I)\setminus\{m\})\cup\{m+s\}.
\]
When $G=H=\mathbb{Z}$ and $I=a+\mathbb{Z}_{\geq0}$ for $a\in\mathbb{Z}$,
deleting $a$ shifts the interval $J(I)=\{a,\ldots,a+s-1\}$ by one.
These intervals
are the weight sets of generic GIT windows \cite[Definition~3.6]{HH24}
and specify the summands in NCCR constructions
\cite[Remark~4.20]{VdB23}. Van den Bergh's construction over an
algebraically closed field of characteristic zero uses $a=0$
\cite[Theorem~8.9]{VdB04a}.
Hara and Hirano use Koszul complexes to realize exchanges between
adjacent GIT windows by iterated Iyama--Wemyss mutations for
generic quasi-symmetric torus representations over an algebraically
closed field of characteristic zero \cite[Theorem~4.14]{HH24}.

For an arbitrary cover $I\supset I^-$ in the upper-set classification,
the intermediate Koszul shifts lie over $J(I)\setminus\{m\}$, the
common part of the two windows.

\begin{Thm}[Theorem \ref{toricmut}]
Let $I$ be a non-trivial upper set in $H$, let $m$ be a minimal
element of $I$, and put $I^-:=I\setminus\{m\}$. With the Koszul
approximations, the passage
\[
M_{q^{-1}(J(I))}\longmapsto M_{q^{-1}(J(I^-))}
\]
is realized up to isomorphism by $l'$ right or $l$ left mutations at
the common summand $M_{q^{-1}(J(I)\setminus\{m\})}$.

Consequently, any two basic reflexive $R$-modules giving toric
NCCRs can be joined by a finite sequence of Iyama--Wemyss mutations
(Corollary~\ref{toricNCCRconn}).
\end{Thm}

The intermediate modules give NCCRs and may have higher-rank
summands. Deleting minimal elements takes both upper sets to their
intersection (Proposition~\ref{muttrans}), giving a common endpoint
for the mutation sequences. Thus all toric NCCRs of $R$ are derived
equivalent, giving the toric case of the non-commutative
Bondal--Orlov conjecture for these rings
\cite[Conjecture~4.6]{VdB04a}\cite[Conjecture~1.2]{SVdBB22}.

Related mutation results in dimension three concern conifold quotients
\cite[Theorem~8.8]{Boc12}, maximal modifying modules formed from
divisorial modules over complete local toric rings of reflexive polygons
in characteristic zero \cite[Theorem~1.2]{Nak19}, and NCCRs of completed
anticanonical cones over smooth complex del Pezzo surfaces
\cite[Theorems~1.1 and 1.2]{NVdB26}.

In the non-pyramidal case, $H=G$ (Proposition~\ref{pyra}), and the
ordered weights determine a periodic quiver $Q$ on the $d$-torus
with vertex set $G/\mathbb{Z}s$ and a cut type. For an arrow $v\to w$
of weight $\vec{z}$ with representatives $\vec{g}_v,\vec{g}_w\in J(I)$,
write $\vec{g}_v+\vec{z}=\vec{g}_w+\varepsilon s$. The two weight
families give $\varepsilon\in\{0,1\}$ and $\varepsilon\in\{0,-1\}$,
respectively, and the arrow belongs to the cut $C$ exactly when
$\varepsilon\ne0$ (Proposition~\ref{cutJcorr}). Each weight family
has at least two members, so in the free quotient of $G$ the
absolute value of each weight is strictly between $0$ and the
image of $s$. Thus each arrow joins distinct vertices. Deleting a
minimal element $m$ as in Theorem~\ref{toricmut} replaces $m$ by
$m+s$ in the window. This increases $\varepsilon$ by one on
arrows out of $m+\mathbb{Z}s$ and decreases it by one on arrows
into that vertex, so precisely these arrows of $Q$ switch their
cut status.

Adapting the height functions of Dramburg and Gasanova
\cite[\S~2]{DG} to the two weight families, Theorem~\ref{cutupcorr}
identifies cuts of the prescribed type with upper sets modulo
translation by multiples of $s$. A shift by a nonzero multiple of
$s$ changes the isomorphism class of the basic module and preserves
its endomorphism algebra up to isomorphism (Proposition~\ref{refhom}).

Over an algebraically closed field, Craw and Quintero V\'elez
present the endomorphism algebra of a finite sum of pairwise
non-isomorphic rank-one reflexive sheaves on a normal affine toric
variety with a torus-fixed point by a quiver of sections with
equal-label relations \cite[\S~2.2]{CQV12}.
An arrow of $Q$ outside $C$ gives an arrow of $Q(C)$ labelled by
its variable. If the arrows out of $v$ for $x_i$ and $x'_j$ both
belong to $C$, they give an arrow labelled $x_ix'_j$ from $\vec{g}_v$ to
$\vec{g}_v+\vec{x}_i+\vec{x}'_j\in J(I)$.

\begin{Thm}[Theorem \ref{dimerreal}]
Suppose the associated polytope is non-pyramidal and let $I$ be a
non-trivial upper set in $G$. Let $C$ be the corresponding cut
and $Q(C)$ the quiver above, with vertices indexed by $J(I)$.
Its arrows, with labels of total degree one or two, represent exactly
the monomial maps of positive total degree between the summands of
$M_{J(I)}$ that cannot be factored into two such maps through one of
these summands. Moreover,
\[
\End_R(M_{J(I)})\cong kQ(C)/I(C),
\]
where $I(C)$ is generated by differences of parallel paths with
the same monomial label.
\end{Thm}

Nonconsecutive summand sets occur for one-dimensional tori
\cite[\S~6]{SVdBB22}. In Example~\ref{exnoninterval}, the weights
$(2,3,-2,-3)$ give $s=5$. Replacing $0$ by $5$ in the window
$\{0,2,3,4,6\}$ makes the $y$-arrow from the vertex represented
by $2$ a retained arrow, so the mixed loop labelled $yw$ at that
vertex disappears.

A dimer model on the two-torus is a bipartite tiling of the torus
\cite[\S~2.1]{Bro}. For $d=2$, $Q(C)$ is the dual quiver of such a
tiling. Over $\mathbb{C}$, Bocklandt's algorithm recovers dimer
models from toric NCCRs in dimension three \cite[\S~7]{Boc12}.
His four-dimensional example uses homotopy on the three-torus
together with $R$-charge \cite[\S~9.2]{Boc12}. Craw and Quintero
V\'elez extend the superpotential formalism and the notion of
algebraic consistency to toric algebras of arbitrary dimension
\cite[\S~3 and \S~6]{CQV12}.

The cut determined by the window gives the arrows of $Q(C)$ with
linear or quadratic monomial labels. Each arrow lies on a cycle
labelled by the anticanonical monomial $D$, the product of all the variables
(Corollary~\ref{dimercycles}). Two parallel paths agree modulo
$I(C)$ precisely when they are homotopic relative to their endpoints
on the $d$-torus and their labels have the same total degree
(Lemma~\ref{dimerhomotopy}). These properties make $Q(C)$ a
higher-dimensional analogue of the dual quiver of a dimer model.

In the non-pyramidal case, $|Q(C)_0|=|G/\mathbb{Z}s|$ counts the
summands of $M_{J(I)}$. For any associated polytope $P$, including
pyramids, each coset of $\mathbb{Z}s$ in $H$ contributes $|\ker q|$
summands to the basic module $M_{q^{-1}(J(I))}$. By the volume
formula in Appendix~\ref{appvolume}, the number of summands is
\[
\#q^{-1}(J(I))=|\ker q|\,|H/\mathbb{Z}s|=d!\vol(P).
\]
This proves Van den Bergh's conjectured equality with normalized
volume for toric NCCRs in class-group rank one
(Theorem~\ref{volrk}) \cite[Conjecture~3.8]{VdB23}.

Appendix~\ref{apptorsion} treats finite divisor class groups. The
sum of one divisorial module from each class is the unique basic
reflexive module giving a toric NCCR, up to $R$-module isomorphism.

\subsection*{Conventions}
Throughout this paper, $k$ denotes an arbitrary field. For an abelian group $G$ and a $G$-graded ring $A$, let $\mod^G\!A,\refl^G\!A$ and $\proj^G\!A$ denote the categories of finitely generated $G$-graded right $A$-modules, $G$-graded reflexive right $A$-modules and finitely generated $G$-graded projective right $A$-modules respectively.
For a module $U$, let $\add U$ denote the full subcategory of direct summands of finite direct sums of copies of $U$.

\subsection*{AI Declaration}
The author used OpenAI Codex (model \texttt{gpt-6.1-sol}) in
revising the manuscript. Codex assisted in checking and revising
mathematical statements and parts of proofs, including the proof
of Theorem~\ref{upNCCRcorr} and the lattice argument for
Lemma~\ref{dimerhomotopy}. It also helped revise and check the
computations in Example \ref{example:non-pyramid} and edit the \LaTeX{} source of the
quiver diagrams. The author used Codex to check references and
improve the wording and organisation of the text, especially in
the abstract and introduction. The author takes full responsibility
for the content of the manuscript.

\section*{Acknowledgements}
The author is grateful to Wahei Hara, Osamu Iyama and Koji Matsushita for fruitful discussions. This work was supported by the WINGS-FMSP program at the Graduate School of Mathematical Sciences, the University of Tokyo, and JSPS KAKENHI Grant Number JP25KJ0818.

\section{Combinatorics on upper sets}\label{combup}

In this section, we collect some basic properties of upper sets in partially ordered sets on which the group $\mathbb{Z}$ of integers acts. Let $X$ be a non-empty partially ordered set.

\begin{Def}
We call a subset $I\subseteq X$ an {\it upper set} if for all $x\in I$ and $y\in X$ with $x\leq y$, $y\in I$ holds. An upper set $I\subseteq X$ is called {\it non-trivial} if $I\neq\emptyset,X$.

Put $\I_X:=\{I\subseteq X\colon\text{non-trivial upper set}\}$.
\end{Def}

Assume that $\mathbb{Z}$ acts on the set $X$ and satisfies the following conditions. Here, we write
\[x+np:=n\cdot x\]
for $n\in\mathbb{Z}$. Here, $p$ is just a symbol.
\begin{enumerate}
\item[(A1)] $x<x+p$ holds for all $x\in X$.
\item[(A2)] $x\leq y$ implies $x+np\leq y+np$ for all $x,y\in X$ and $n\in\mathbb{Z}$.
\item[(A3)] For any $x,y\in X$, there exists $n\in\mathbb{Z}$ such that $x+np\geq y$ holds.
\end{enumerate}

Thanks to the condition (A3), we can prove the following lemma.

\begin{Lem}\label{boundary}
For $I\in\I_X$ and $x\in X$, there exists a unique $n_0\in\mathbb{Z}$ so that the following conditions are equivalent for $n\in\mathbb{Z}$.
\begin{enumerate}
\item $x+np\in I$
\item $n\geq n_0$
\end{enumerate}
\end{Lem}
\begin{proof}
Take $y\in I$ and $z\in I^c$. By (A3), there exists $l,m\in\mathbb{Z}$ such that $x+lp\geq y$ and $x-mp\leq z$ hold. This means $x+lp\in I$ and $x-mp\notin I$. Thus such an $n_0$ exists uniquely.
\end{proof}

For $I\in\I_X$, the \emph{window} of $I$ is
$J(I):=I\cap(I^c+p)$. By Lemma~\ref{boundary}, it consists of the
first element of $I$ in each $\mathbb{Z}$-orbit.

Define a set $\widetilde{\J}_X$ and $\J_X$ consisting of subsets of $X$ as
\[\widetilde{\J}_X:=\{J\subseteq X\mid\text{For any }x,y\in J,\text{ we have }x\ngeq y+p.\}\text{ and}\]
\[\J_X:=\{J\in\widetilde{\J}_X\colon\text{maximal with respect to inclusion}\}\subseteq\widetilde{\J}_X.\]

\begin{Thm}\label{upJX}
The maps
\[J(-)\colon\I_X\rightleftarrows\J_X :I(-)\]
are inverse bijections, where
$I(J):=\{x\in X\mid x\geq y\text{ for some }y\in J\}$.
\end{Thm}

We make some preparations.

\begin{Prop}\label{charJX}
For a subset $J\subseteq X$, the following conditions are equivalent.
\begin{enumerate}
\item $J\in\widetilde{\J}_X$
\item There exists $I\in\I_X$ such that $J\subseteq I\cap(I^c+p)$ holds.
\end{enumerate}
\end{Prop}
\begin{proof}
(1)$\Rightarrow$(2) If $J=\emptyset$, choose $x\in X$ and take the upper set $\{y\in X\mid y\geq x\}$, which omits $x-p$. Otherwise, put $I:=\{y\in X\mid y\geq x\text{ for some }x\in J\}$. For each $x\in J$, we have $x\in I$ and $x-p\notin I$, since $x-p\geq y$ for some $y\in J$ would contradict (1). Thus $I\in\I_X$ and $J\subseteq I\cap(I^c+p)$.

(2)$\Rightarrow$(1) Take $x,y\in I\cap(I^c+p)$. Then we have $x-p\notin I$ and $y\in I$. Since $I$ is an upper set, we obtain $x-p\ngeq y$. Thus $I\cap(I^c+p)\in\widetilde{\J}_X$ holds.
\end{proof}

\begin{Prop}\label{charmaxJX}
For $J\in\widetilde{\J}_X$, the following conditions are equivalent.
\begin{enumerate}
\item $J\in\J_X$
\item $J\subseteq X$ is a complete set of representatives for the action of $\mathbb{Z}$ on $X$.
\item There exists $I\in\I_X$ such that $J=I\cap(I^c+p)$ holds.
\end{enumerate}
\end{Prop}
\begin{proof}
(1)$\Rightarrow$(3) follows from Proposition \ref{charJX}.

(3)$\Rightarrow$(2) For $I$ and $x\in X$, take $n_0\in\mathbb{Z}$ as in Lemma \ref{boundary}. Then we have $x+np\in I\cap(I^c+p)\Leftrightarrow n=n_0$ for $n\in\mathbb{Z}$.

(2)$\Rightarrow$(1) A proper enlargement of $J$ contains distinct elements $x,x'$ in the same orbit. After interchanging them if necessary, we have $x=x'+np$ for some $n\geq1$, so $x\geq x'+p$. Hence no proper enlargement lies in $\widetilde{\J}_X$.
\end{proof}

Now we can prove Theorem \ref{upJX}.

\begin{proof}[Proof of Theorem \ref{upJX}]
The well-definedness of $J(-)$ follows from Proposition \ref{charmaxJX}. In the proof of (1)$\Rightarrow$(2) of Proposition \ref{charJX}, we showed $J\subseteq J(I(J))$. By the maximality of $J$, we have $J=J(I(J))$. Next, we prove $I(J(I))=I$. $I(J(I))\subseteq I$ is obvious. Take $x\in I$. Take $n_0\in\mathbb{Z}$ as in Lemma \ref{boundary}. Then $x+n_0p\in J(I)$ and $x\geq x+n_0p$ hold. This means $x\in I(J(I))$.
\end{proof}

Next, we introduce the notion of mutation for upper sets.

\begin{Def}\label{defmutup}
Let $I\in\I_X$ and take a minimal element $m\in I$. Then we define the {\it mutation} $\mu_{m}^-(I)$ of $I$ at $m$ as
\[\mu_m^-(I):=I\setminus\{m\}.\]
\end{Def}

\begin{Rem}
Observe that $\mu_m^-(I)\in\I_X$ holds. In addition, we have
\[J(\mu_m^-(I))=J(I)\sqcup\{m+p\}\setminus\{m\}.\]
\end{Rem}

Observe that $\I_X$ has a natural partial order defined by inclusion:
\[I\leq I':\Leftrightarrow I\subseteq I'.\]
Since $I\cap I', I\cup I'\in\I_X$ hold for $I,I'\in\I_X$, $\I_X$ becomes a lattice with this order. We can see that the arrows in the Hasse quiver of $\I_X$ correspond to the mutations.

\begin{Prop}
For $I,I'\in\I_X$ with $I\geq I'$, the following conditions are equivalent.
\begin{enumerate}
\item In the Hasse quiver of $\I_X$, there exists an arrow $I\to I'$.
\item $\sharp(I\setminus I')=1$
\item $I'=\mu_m^-(I)$ holds for some minimal element $m\in I$.
\end{enumerate}
\end{Prop}
\begin{proof}
(1)$\Leftarrow$(2)$\Leftrightarrow$(3) is obvious. We prove (1)$\Rightarrow$(2). Assume there exists an arrow $I\to I'$. Suppose there exist two distinct elements $x,y\in I\setminus I'$. We may assume $x\ngeq y$. Then
\[I'':=I'\cup\{z\in X\mid z\geq y\}\in\I_X\]
satisfies $I>I''>I'$, which is a contradiction.
\end{proof}

We show that under a certain finiteness condition, all non-trivial upper sets in $X$ are connected by iterated mutations.

\begin{Prop}\label{muttrans}
Assume that the action of $\mathbb{Z}$ on $X$ has only finitely many orbits. Let $I,I'\in\I_X$ with $I>I'$. Then $I$ becomes $I'$ after iterated mutations.
\end{Prop}
\begin{proof}
By our assumption, we have $\sharp(I\setminus I')<\infty$. Thus there exists a minimal element $m\in I\setminus I'$. This also gives a minimal element of $I$. Now we have $I>\mu_m^-(I)\geq I'$ and $\sharp(\mu_m^-(I)\setminus I')=\sharp(I\setminus I')-1$. Thus by the induction, we get the assertion.
\end{proof}

We exhibit a corollary for later use.

\begin{Cor}\label{mutgen}
Assume that the action of $\mathbb{Z}$ on $X$ has only finitely many orbits. Assume a subset $Y\subseteq X$ satisfies the following conditions.
\begin{enumerate}
\item For any $I\in\I_X$ and any minimal element $m\in I$, $J(I)\subseteq Y$ holds if and only if $J(\mu_m^-(I))\subseteq Y$ holds.
\item There exists $I_0\in\I_X$ such that $J(I_0)\subseteq Y$ holds.
\end{enumerate}
Then we have $Y=X$.
\end{Cor}
\begin{proof}
Take $x\in X$ and put $I:=\{y\in X\mid y\geq x\}\in\I_X$. By $I_0\cap I\in\I_X$ and Proposition \ref{muttrans}, we can conclude that $x\in J(I)\subseteq Y$ holds.
\end{proof}

\section{Preliminaries on toric singularities}

Let $N\cong\mathbb{Z}^{d+1}$ be a lattice of rank $d+1$ and $M:=\Hom_\mathbb{Z}(N,\mathbb{Z})$ the dual lattice of $N$. Let $\sigma\subseteq N_\mathbb{R}:=N\otimes_\mathbb{Z}\mathbb{R}$ be a strongly convex non-degenerate rational polyhedral cone. Put $R:=k[\sigma^\vee\cap M]$ be the toric singularity defined by $\sigma$. Let $\{\rho_i\}_{i=1}^n$ denote the set of the rays of $\sigma$ and take the minimal generator $v_i\in\rho_i\cap N$. Then we have a natural group homomorphism
\[\phi\colon\mathbb{Z}^n\to N\]
with finite cokernel which sends the $i$-th unit vector to $v_i\in N$. Define an abelian group $G$ by the short exact sequence
\[0\to M\xrightarrow{\phi^*}(\mathbb{Z}^n)^*\to G\to 0.\]
For $1\leq i\leq n$, we write $\vec{x}_i\in G$ for the image of the $i$-th unit vector of $(\mathbb{Z}^n)^*$. Then the polynomial ring $S:=k[x_1,\cdots, x_n]$ can be viewed as a $G$-graded $k$-algebra by $\deg x_i:=\vec{x}_i$. Now we have a $k$-algebra homomorphism $R\to S$ induced by $\phi$. Then it is easy to see that this homomorphism gives an isomorphism
\[R\xrightarrow[\cong]{}S_0.\]
Thus for $\vec{g}\in G$, we can view $S_{\vec{g}}$ as an $R$-module.

\subsection{Divisorial modules}

It is well-known that $G$ is isomorphic to the divisor class group $\Cl(R)$ of $R$. If we view $\Cl(R)$ as the group of isoclasses of divisorial $R$-modules, then we have the following explicit description of $G\cong\Cl(R)$.

\begin{Thm}\label{clgptoric}
The group homomorphism
\[G\to\Cl(R);\vec{g}\mapsto[S_{\vec{g}}]\]
is an isomorphism.
\end{Thm}

Since we cannot find a good reference dealing with arbitrary field $k$, we include a complete proof for the convenience of the reader.

\begin{Prop}\label{refhom}
For $\vec{g},\vec{h}\in G$, the natural $R$-module homomorphism
\[S_{\vec{h}-\vec{g}}\to\Hom_R(S_{\vec{g}},S_{\vec{h}});s\mapsto s\cdot-\]
is an isomorphism.
\end{Prop}
\begin{proof}
The injectivity is obvious. We prove the surjectivity. Take an $R$-module homomorphism $0\neq f\colon S_{\vec{g}}\to S_{\vec{h}}$. Observe that by multiplying elements of $S_{-\vec{g}}\backslash\{0\}$ and $S_{-\vec{h}}\backslash\{0\}$, $S_{\vec{g}}$ and $S_{\vec{h}}$ are isomorphic to ideals of $R$. Thus we can conclude that there exist homogeneous elements $s,s'\in S\backslash\{0\}$ such that $f=\dfrac{s'}{s}\cdot-$ holds. Therefore it is enough to show that $\dfrac{s'}{s}\in S$ holds.

We may assume that $s,s'\in S\backslash\{0\}$ are coprime to each other. Fix $1\leq i_0\leq n$. Since $v_{i_0}\in\rho_{i_0}\cap N$ is the minimal generator, there exists $m\in M$ satisfying $m(v_{i_0})=1$. This means $G=\sum_{i\neq i_0}\mathbb{Z}\vec{x}_i$. On the other hand, since $\{\rho_i\}_{i=1}^n$ is the set of the rays of $\sigma$, there exists $m'\in M$ satisfying
$m'(v_i)\left\{
\begin{array}{ll}
=0 & (i=i_0)\\
>0 & (i\neq i_0)
\end{array}
\right.$. This means $\sum_{i\neq i_0}m'(v_i)\vec{x}_i=0$ holds. Combining these two facts, we can conclude that there exists $a_i\in\mathbb{Z}_{\geq0}$ for $i\neq i_0$ such that $\sum_{i\neq i_0}a_i\vec{x}_i=\vec{g}$ holds. Since $f(\prod_{i\neq i_0}x_i^{a_i})\in S$, we have
\[s'\prod_{i\neq i_0}x_i^{a_i}\in sS.\]
Thus if $s$ has a prime divisor, then it should be an associate of one of the elements of $\{x_i\}_{i\neq i_0}$. Since this holds for any $1\leq i_0\leq n$, we can conclude that $s$ has no prime divisor.
\end{proof}

\begin{Prop}\label{degreeZeroEquivalence}
The functor
\[(-)_0\colon\refl^G\!S\to\refl R\]
is a categorical equivalence.
\end{Prop}
\begin{proof}
First, we see the well-definedness. By Proposition \ref{refhom}, for any $\vec{g}\in G$, the natural homomorphism
\[S_{\vec{g}}\to\Hom_R(\Hom_R(S_{\vec{g}},R),R)\]
is an isomorphism. Thus $S(\vec{g})_0=S_{\vec{g}}\in\refl R$ holds. Next, for any $Y\in\refl^G\!S$, there exists an exact sequence
\[0\to Y\to P\to Q\]
in $\mod^G\!S$ where $P,Q\in\proj^G\!S$. Thus we have an exact sequence 
\[0\to Y_0\to P_0\to Q_0\]
in $\mod R$. Since $P_0,Q_0\in\refl R$ holds, we obtain $Y_0\in\refl R$.

Next, we show the full faithfulness. Take $X,Y\in\refl^G\!S$. First, assume $X\in\proj^G\!S$. Take an exact sequence $0\to Y\to P\to Q$ with $P,Q\in\proj^G\!S$. Consider the following commutative diagram.
\[
\xymatrix{
0 \ar[r] & \Hom_S^G(X,Y) \ar[r] \ar[d] & \Hom_S^G(X,P) \ar[r] \ar[d] & \Hom_S^G(X,Q) \ar[d]\\
0 \ar[r] & \Hom_R(X_0,Y_0) \ar[r] & \Hom_R(X_0,P_0) \ar[r] & \Hom_R(X_0,Q_0)
}
\]
The two horizontal sequences are exact. By Proposition \ref{refhom}, $(-)_0\colon\proj^G\!S\to\refl R$ is fully faithful. Thus the two vertical arrows on the right are isomorphisms. Therefore the leftmost vertical arrow is an isomorphism. Second, consider arbitrary $X\in\refl^G\!S$. Take a projective presentation $Q'\to P'\to X\to 0$ in $\mod^G\!S$. Consider the following commutative diagram.
\[
\xymatrix{
0 \ar[r] & \Hom_S^G(X,Y) \ar[r] \ar[d] & \Hom_S^G(P',Y) \ar[r] \ar[d] & \Hom_S^G(Q',Y) \ar[d]\\
0 \ar[r] & \Hom_R(X_0,Y_0) \ar[r] & \Hom_R(P'_0,Y_0) \ar[r] & \Hom_R(Q'_0,Y_0)
}
\]
The two horizontal sequences are exact. By the discussion above, the two vertical arrows on the right are isomorphisms. Therefore the leftmost vertical arrow is an isomorphism.

Finally, we show essential surjectivity. Take $M\in\refl R$. Then we have an exact sequence $0\to M\to R^{\oplus m}\xrightarrow{f} R^{\oplus n}$ for some $m,n\geq0$. Then there exists $g\colon S^{\oplus m}\to S^{\oplus n}$ in $\mod^G\!S$ such that $g_0=f$ holds. Then we have $\Ker g\in \refl^G\!S$ and $(\Ker g)_0\cong M$.
\end{proof}

\begin{Rem}
We can show the following stronger condition than the full faithfulness of $(-)_0\colon\refl^G\!S\to\refl R$ in the same way: for $X\in\mod^G\!S$ and $Y\in\refl^G\!S$, the natural map
\[\Hom_S^G(X,Y)\to\Hom_R(X_0,Y_0)\]
is bijective.
\end{Rem}

\begin{proof}[Proof of Theorem \ref{clgptoric}]
By Proposition \ref{refhom}, we have
\[[S_{\vec{h}-\vec{g}}]=[\Hom_R(S_{\vec{g}},S_{\vec{h}})]=[S_{\vec{h}}]-[S_{\vec{g}}].\]
Thus the map is a group homomorphism. Since $S^\times=k^\times$, the shifts $S(\vec{g})$ are pairwise non-isomorphic as graded modules. Proposition~\ref{degreeZeroEquivalence} gives injectivity.

For surjectivity, let $M$ be divisorial and take $X\in\refl^G\!S$ with $X_0\cong M$ by Proposition~\ref{degreeZeroEquivalence}. Put $K=\Frac R$ and $L=S\otimes_R K$. Each component $L_{\vec{g}}$ contains a unit, since the proof of Proposition~\ref{refhom} gives nonzero elements of $S_{\vec{g}}$ and $S_{-\vec{g}}$ whose product lies in $R\setminus\{0\}$. A basis element of $(X\otimes_S L)_0\cong K$ then freely generates $X\otimes_S L$ over $L$. Hence $X$ has rank one.

By reflexivity, $X$ embeds into a finite sum of graded shifts of $S$. A nonzero component $X\to S(\vec{a})$ is injective because $X$ is torsion-free of rank one. Its image in $S$ is a homogeneous reflexive ideal, hence principal because $S$ is factorial. A homogeneous component of a generator $f$ with the same total degree as $f$ must be a nonzero scalar multiple of $f$. Thus $f$ is homogeneous, $X\cong S(\vec{g})$ for some $\vec{g}\in G$, and $M\cong S_{\vec{g}}$.
\end{proof}

For a finite subset $J\subseteq G$, put $M_J:=\bigoplus_{\vec{g}\in J}S_{\vec{g}}$. The degrees are recovered from its additive closure by the formula
\begin{equation}\label{addDegrees}
J=\{\vec{h}\in G\mid S_{\vec{h}}\in\add M_J\}.
\end{equation}
Indeed, if $S_{\vec{h}}\in\add M_J$ with $\vec{h}\notin J$, its identity factors through a finite sum of copies of $M_J$. Proposition \ref{refhom} then gives the contradiction
\[
1\in\sum_{\vec{g}\in J}S_{\vec{g}-\vec{h}}S_{\vec{h}-\vec{g}}
\subseteq\{f\in R\mid f(0)=0\}.
\]

The multiplicities in a divisorial decomposition
$M\cong\bigoplus_{\vec{g}\in G} S_{\vec{g}}^{\oplus a_{\vec{g}}}$
are determined by the isomorphism class of $M$. By
Proposition~\ref{degreeZeroEquivalence}, the corresponding graded
$S$-module is $P=\bigoplus_{\vec{g}\in G}S(\vec{g})^{\oplus a_{\vec{g}}}$
up to graded isomorphism. For the ideal $\mathfrak{n}\subseteq S$
generated by the variables, the component of $P/\mathfrak{n}P$ of
degree $-\vec{g}$ has dimension $a_{\vec{g}}$ over $k$. In
particular, being basic is independent of the divisorial
decomposition.

Finally, we recall the calculation of the local cohomologies of divisorial $R$-modules.

\begin{Prop}\label{lcohtoric}
Let $\mathfrak{m}$ be the maximal ideal of polynomials in $R$ with zero constant term. The toric \v{C}ech complex has terms
\[
C^r=\bigoplus_{\substack{\tau\text{ a face of }\sigma\\ \dim\tau=d+1-r}}
S_{\prod_{\rho_i\nsubseteq\tau}x_i}
\qquad (0\le r\le d+1)
\]
and differentials given by localization maps with the signs of the augmented cellular boundary of a cross-section of $\sigma$. For $\vec{g}\in G$, the degree-$\vec{g}$ subcomplex computes $H^r_{\mathfrak{m}}(S_{\vec{g}})$.
\end{Prop}
\begin{proof}
For each face $\tau$, choose $m_\tau\in\sigma^\vee\cap M$ vanishing exactly on $\tau$. In $S$, the monomial $\chi^{m_\tau}$ has positive exponent precisely on the variables corresponding to rays outside $\tau$. Thus
\[
\bigl(S_{\prod_{\rho_i\nsubseteq\tau}x_i}\bigr)_{\vec{g}}
\cong R[(\chi^{m_\tau})^{-1}]\otimes_R S_{\vec{g}}.
\]
These are the face localizations of $S_{\vec{g}}$, so the assertion follows from \cite[6.2.5]{BrH98}.
\end{proof}

\subsection{Relations among $G$, $\sigma$ and $R$}\label{toricGradingData}

In the proof of Proposition \ref{refhom}, we see that the following two conditions are satisfied.
\begin{enumerate}
\item For each $1\leq i_0\leq n$, we have $\sum_{i\neq i_0}\mathbb{Z}\vec{x}_i=G$.
\item For each $1\leq i_0\leq n$, there exists $m_i>0$ for $i\neq i_0$ such that $\sum_{i\neq i_0} m_i\vec{x}_i=0$ holds.
\end{enumerate}
Conversely, assume that we are given a finitely generated abelian group $G$ and $\vec{x}_1,\cdots,\vec{x}_n\in G$ satisfying (1) and (2). Then we have a natural group epimorphism $(\mathbb{Z}^n)^*\to G$ sending the $i$-th unit vector to $\vec{x}_i$. Define a free abelian group $M$ by the following short exact sequence.
\[0\to M\xrightarrow{\psi}(\mathbb{Z}^n)^*\to G\to0\]
The images $v_i=\psi^*(e_i)$ span $N_\mathbb{R}$. Condition (1) makes each $v_i$ primitive, and (2) provides a functional vanishing on $v_i$ and positive on all other $v_j$. Thus the $v_i$ are the primitive generators of the distinct rays of a strongly convex cone $\sigma$.

The canonical class is $-\sum_{i=1}^n\vec{x}_i$. Hence $R$ is Gorenstein exactly when $\sum_i\vec{x}_i=0$, or equivalently when there is an $m\in M$ with $m(v_i)=1$ for all $i$. The height-one section $P=\sigma\cap\{m=1\}$ is then a $d$-dimensional lattice polytope with vertices $v_i$. Conversely, homogenizing the vertices of a lattice polytope gives the primitive rays of a Gorenstein cone. With ordered vertices, these constructions are inverse up to affine lattice equivalences preserving the order and group isomorphisms preserving the ordered weights. The non-empty faces of $P$ correspond to the positive-dimensional faces of $\sigma$, giving the following form of Proposition \ref{lcohtoric}.

\begin{Prop}\label{lcohGortoric}
Let $P$ be a $d$-dimensional convex lattice polytope and $R$ be the Gorenstein toric ring defined as above. For $c=(c_i)_{i=1}^n\in\mathbb{Z}^n$, put
\[\Delta_c:=\{I\subseteq\{1,\cdots,n\}\mid\Conv\{v_i\}_{i\in I}\text{ is a face of }P\text{ and }c_i\geq0\text{ holds for all }i\in I\}\]
and define a subspace $X_c\subseteq P$ as
\[X_c:=\bigcup_{I\in\Delta_c}\Conv\{v_i\}_{i\in I}.\]
For $\vec{g}\in G$, we have
\[H^r_{\mathfrak{m}}(S_{\vec{g}})\cong\bigoplus_{\substack{c\in\mathbb{Z}^n \\ \sum_ic_i\vec{x}_i=\vec{g}}}\tilde{H}_{d-r}(X_c;k),\]
where $\tilde{H}_{d-r}(X_c;k)$ denotes reduced homology with coefficients in $k$, with the convention $\tilde{H}_{-1}(\varnothing;k)=k$.
\end{Prop}
\begin{proof}
In exponent degree $c$, the summand at $\tau$ in Proposition~\ref{lcohtoric} is $kx^c$ if $c_i\geq0$ for every ray $\rho_i\subseteq\tau$, and zero otherwise. The surviving positive-dimensional faces index the cells of $X_c$, and the zero face gives the augmentation. The incidence differential thus identifies this component of $C^\bullet$ with the augmented cellular chains of $X_c$, with $C^r$ in chain degree $d-r$. Summing the homology over the exponents of $G$-degree $\vec{g}$ gives the formula.
\end{proof}

We introduce the following lemma here which will be used later.

\begin{Lem}\label{torscont}
Take $1\leq i_0\leq n$ and assume $\vec{x}_{i_0}\in G$ is torsion. For any $c=(c_i)_{i=1}^n\in\mathbb{Z}^n$ with $c_{i_0}\geq0$, $X_c$ is contractible.
\end{Lem}
\begin{proof}
Take $I\in\Delta_c$ with $i_0\notin I$. Then there exists $s_i\in\mathbb{Z}$ for $1\leq i\leq n$ such that $s_i\left\{
\begin{array}{ll}
=0 & (i\in I)\\
>0 & (i\notin I)
\end{array}
\right.$ and $\sum_{i=1}^ns_i\vec{x}_i=0$ hold. Take $b\geq1$ with $b\vec{x}_{i_0}=0$. Put $t_i:=\left\{
\begin{array}{ll}
bs_i & (i\neq i_0)\\
0 & (i=i_0)
\end{array}
\right.$. Then we have $t_i\left\{
\begin{array}{ll}
=0 & (i\in I\sqcup{\{i_0\}})\\
>0 & (i\notin I\sqcup{\{i_0\}})
\end{array}
\right.$ and $\sum_{i=1}^nt_i\vec{x}_i=0$. Thus $I\sqcup{\{i_0\}}\in\Delta_c$ holds. Therefore $X_c$ is homotopic to the point $\{v_{i_0}\}$.
\end{proof}

Suppose $G$ has rank one and put $\pi\colon G\to G/G_{\rm{tors}}\cong\mathbb{Z}$. Under (1), condition (2) is equivalent to having both signs among the images of the weights after deleting any one weight, that is,
\[\#\{1\leq i\leq n\mid\pi(\vec{x}_i)>0\}\geq2\text{ and }\#\{1\leq i\leq n\mid\pi(\vec{x}_i)<0\}\geq2\]
Indeed, the remaining images generate $\mathbb{Z}$, and a positive relation among them exists exactly when both signs occur. Multiplying such a relation by the exponent of $G_{\rm tors}$ gives a relation in $G$. In this case $n=d+2$.

\subsection{Pyramids}

In this subsection, we give a remark on pyramids.

\begin{Def}
Let $P\subseteq\mathbb{R}^d$ be a $d$-dimensional convex lattice polytope with vertex set $P_0=\{v_i\}_{i=1}^n$. $P$ is called a {\it pyramid} if there exists an affine hyperplane $H\subseteq\mathbb{R}^d$ and a vertex $v_{i_0}\in P_0$ such that
\[v_i\left\{
\begin{array}{ll}
\in H & (i\neq i_0)\\
\notin H & (i=i_0)
\end{array}
\right.\]
holds for $1\le i\le n$. Roughly speaking, this condition is equivalent to saying that $P$ is obtained by adding a vertex to a $(d-1)$-dimensional convex lattice polytope along a new dimension.
\end{Def}

The following proposition follows from the definition immediately, which gives a criterion for $P$ to be a pyramid in terms of our group $G$.

\begin{Prop}\label{pyra}
For a lattice polytope $P$, the following conditions are equivalent.
\begin{enumerate}
\item $P$ is a pyramid.
\item There exists $1\le i\le n$ such that $\vec{x}_i\in G$ is torsion.
\end{enumerate}
\end{Prop}

\section{Preliminaries on NCCRs}

In this section, we recall a definition and basic properties of NCCRs. The notion of NCCRs was introduced in \cite{VdB04a} as a virtual space of crepant resolutions. For connections with algebraic geometry, see \cite[6.3.1]{VdB04a}.

\begin{Def}\cite[4.1]{IW}\cite[4.1]{VdB04a}
Let $R$ be a Gorenstein normal domain with $\dim R\geq2$ and $M$ a reflexive $R$-module. Put $\Gamma:=\End_R(M)$.
\begin{enumerate}
\item $M$ is called {\it modifying} if $\Gamma$ is maximal Cohen--Macaulay as an $R$-module.
\item $\Gamma$ is called a {\it non-commutative crepant resolution (NCCR)} of $R$ if $M$ is non-zero and modifying and $\gl\Gamma_P<\infty$ holds for all $P\in\Spec R$.
\end{enumerate}
\end{Def}

If $\add U=\add V$, then $\End_R(U)$ and $\End_R(V)$ are Morita
equivalent over $R$. Each is an $R$-module summand of a matrix ring
over the other, so the NCCR property is preserved.

\begin{Def}
Let $R$ be a Gorenstein toric ring with $\dim R\geq2$. A reflexive
$R$-module $M$ {\it gives a toric NCCR} of $R$ if it is a direct sum
of divisorial modules and $\End_R(M)$ is an NCCR of $R$. In this
case $\End_R(M)$ is called a {\it toric NCCR} of $R$.
\end{Def}

In \cite{IW}, the operation called mutations was introduced for NCCRs.

\begin{Def}\cite[1.21]{IW}
Let $R$ be a Gorenstein normal domain with $\dim R\geq2$ and $M$ a modifying $R$-module. Let $0\ncong N\in\add M$. Take a right $(\add N)$-approximation $N_0\to M$ of $M$ and a right $(\add N^*)$-approximation $N_1^*\to M^*$ of $M^*$ where $(-)^*=\Hom_R(-,R)$. Take their kernels.
\[0\to K_0\to N_0\to M,\ 0\to K_1\to N_1^*\to M^*\]
We define the {\it right mutation} of $M$ at $N$ to be $\mu_N^+(M):=N\oplus K_0$ and the {\it left mutation} of $M$ at $N$ to be $\mu_N^-(M):=N\oplus K_1^*$.
Their additive closures are independent of the chosen approximations \cite[6.3]{IW}.
\end{Def}

This operation satisfies the following desirable properties.

\begin{Thm}\cite[6.8,6.10]{IW}
Let $R$ be an equi-codimensional Gorenstein normal domain with $\dim R\geq2$ and $M$ a modifying $R$-module. Let $0\ncong N\in\add M$.
\begin{enumerate}
\item $\End_R(M),\End_R(\mu_N^+(M))$ and $\End_R(\mu_N^-(M))$ are derived equivalent to one another.
\item $\mu_N^+(M)$ and $\mu_N^-(M)$ are modifying modules.
\item If $M$ gives an NCCR, then so do $\mu_N^+(M)$ and $\mu_N^-(M)$.
\end{enumerate}
\end{Thm}

\section{NCCRs of toric singularities with divisor class group of rank one}

Let $R$ be a Gorenstein toric singularity with $\rk\Cl(R)=1$. By Section~\ref{toricGradingData}, such an $R$ can be described by the following grading data. Let $G$ be a finitely generated abelian group of rank one. Let $l\geq1, l'\geq1$ and $d\geq l+l'$. Put $\pi\colon G\to G/G_{\rm{tors}}\cong\mathbb{Z}$ and take elements $\vec{x}_0,\cdots,\vec{x}_l,\vec{x}'_0,\cdots,\vec{x}'_{l'},\vec{y}_1,\cdots,\vec{y}_{d-l-l'}\in G$ satisfying the following three conditions. For $1\le j\le d+2$, write
\[\vec{z}_j:=\left\{
\begin{array}{ll}
\vec{x}_{j-1} & (1\leq j\leq l+1)\\
\vec{x}'_{j-l-2} & (l+2\leq j\leq l+l'+2)\\
\vec{y}_{j-l-l'-2} & (l+l'+3\leq j\leq d+2)
\end{array}
\right..\]
\begin{enumerate}
\item[(G1)] $\sum_{j=1}^{d+2}\mathbb{Z}\vec{z}_j=G$
\item[(G2)] $\sum_{j=1}^{d+2}\vec{z}_j=0$
\item[(G3)] $\pi(\vec{x}_i)>0,\pi(\vec{x}'_{i'})<0,\pi(\vec{y}_{i''})=0\ (0\le i\le l,0\le i'\le l',1\le i''\le d-l-l')$
\end{enumerate}

In this notation, $S=k[x_0,\cdots,x_l,x'_0,\cdots,x'_{l'},y_1,\cdots,y_{d-l-l'}]$ and $R=S_0$. We remark that by Proposition \ref{pyra}, $d=l+l'$ holds if and only if the corresponding $d$-dimensional lattice polytope $P$ is not a pyramid.

\subsection{When are divisorial modules Cohen--Macaulay?}

Recall that we have a group isomorphism $G\xrightarrow{\cong}\Cl(R);\vec{g}\mapsto[S_{\vec{g}}]$ (Theorem \ref{clgptoric}). In this subsection, we give an equivalent condition for $S_{\vec{g}}$ to be maximal Cohen--Macaulay as an $R$-module. Put $H:=G/(\sum_{i''=1}^{d-l-l'}\mathbb{Z}\vec{y}_{i''})$ and let $q\colon G\to H$ be the natural surjection. Define
\[H_{\geq0}:=\sum_{i=0}^l\mathbb{Z}_{\geq0}q(\vec{x}_i)+\sum_{i'=0}^{l'}\mathbb{Z}_{\geq0}q(-\vec{x}'_{i'})\subseteq H.\]
Then we can define a partial order on $H$ as
\[h_1\geq h_2:\Leftrightarrow h_1-h_2\in H_{\geq0}\text{ for }h_1,h_2\in H.\]
Put $s:=\sum_{i=0}^lq(\vec{x}_i)=\sum_{i'=0}^{l'}q(-\vec{x}'_{i'})\in H$.

Now we can give the following characterization of Cohen--Macaulayness of divisorial modules as a generalization of \cite[7.8]{Sta}. For $a=(a_i)_{i=0}^l\in\mathbb{Z}^{l+1}$, we write
\[a\cdot\vec{x}:=\sum_{i=0}^la_i\vec{x}_i\in G.\]
We also define $a'\cdot\vec{x}'$ and $b\cdot\vec{y}$ for $a'\in\mathbb{Z}^{l'+1}$ and $b\in\mathbb{Z}^{d-l-l'}$ in the same way.

\begin{Thm}\label{divCM}
Let $R$ and $G$ be as above. For $\vec{g}\in G$, the following conditions are equivalent.
\begin{enumerate}
\item $S_{\vec{g}}$ is maximal Cohen--Macaulay as an $R$-module.
\item $q(\vec{g})\ngeq s$ and $q(\vec{g})\nleq -s$ hold.
\end{enumerate}
\end{Thm}

\begin{proof}
Note that (1) is equivalent to $H^r_{\mathfrak{m}}(S_{\vec{g}})=0$ holds for $0\leq r\leq d$. Thus first, for $c=(a,a',b)\in\mathbb{Z}^{d+2}$ where $a=(a_i)_{i=0}^{l}\in\mathbb{Z}^{l+1},a'=(a'_{i'})_{i'=0}^{l'}\in\mathbb{Z}^{l'+1}$ and $b=(b_{i''})_{i''=1}^{d-l-l'}\in\mathbb{Z}^{d-l-l'}$, we compute $\tilde{H}_r(X_c;k)$ in the notation of Proposition \ref{lcohGortoric} by using the computations of $X_c$ in Lemma \ref{caltop}. Then by Proposition \ref{lcohGortoric}, we can conclude that $S_{\vec{g}}$ is {\it not} maximal Cohen--Macaulay if and only if there exists $c=(a,a',b)\in\mathbb{Z}^{d+2}$ satisfying $a\cdot\vec{x}+a'\cdot\vec{x}'+b\cdot\vec{y}=\vec{g}$ and either the following (6) or (7).
\begin{enumerate}
\item[(6)] $a\in\mathbb{Z}_{\ge0}^{l+1}, a'\in\mathbb{Z}_{<0}^{l'+1}$ and $b\in\mathbb{Z}_{<0}^{d-l-l'}$
\item[(7)] $a\in\mathbb{Z}_{<0}^{l+1}, a'\in\mathbb{Z}_{\ge0}^{l'+1}$ and $b\in\mathbb{Z}_{<0}^{d-l-l'}$
\end{enumerate}
If there exists $\vec{c}\in\mathbb{Z}^{d+2}$ satisfying $a\cdot\vec{x}+a'\cdot\vec{x}'+b\cdot\vec{y}=\vec{g}$ and (6), then we have $q(\vec{g})\geq s$ in $H$. Conversely, assume $q(\vec{g})\geq s$ holds. Then there exist $a\in\mathbb{Z}^{l+1}_{\geq0}$ and $a'\in\mathbb{Z}^{l'+1}_{<0}$ such that $\vec{g}-a\cdot\vec{x}-a'\cdot\vec{x}'\in\sum_{i''=1}^{d-l-l'}\mathbb{Z}\vec{y}_{i''}$ holds. Since $\vec{y}_1,\cdots,\vec{y}_{d-l-l'}\in G$ are torsion, there exists $b\in\mathbb{Z}^{d-l-l'}_{<0}$ such that $\vec{g}-a\cdot\vec{x}-a'\cdot\vec{x}'=b\cdot\vec{y}$ holds. In the same way, we can prove that there exists $\vec{c}\in\mathbb{Z}^{d+2}$ satisfying $a\cdot\vec{x}+a'\cdot\vec{x}'+b\cdot\vec{y}=\vec{g}$ and (7) if and only if $q(\vec{g})\leq-s$ holds. Therefore we get the assertion.
\end{proof}

\begin{Lem}\label{caltop}
Let $c=(a,a',b)\in\mathbb{Z}^{d+2}$ where $a=(a_i)_{i=0}^{l}\in\mathbb{Z}^{l+1},a'=(a'_{i'})_{i'=0}^{l'}\in\mathbb{Z}^{l'+1}$ and $b=(b_{i''})_{i''=1}^{d-l-l'}\in\mathbb{Z}^{d-l-l'}$.
\begin{enumerate}
\item If $b\notin\mathbb{Z}_{<0}^{d-l-l'}$ holds, then $X_c$ is contractible.
\end{enumerate}
In what follows, we assume $b\in\mathbb{Z}_{<0}^{d-l-l'}$ holds.
\begin{enumerate}
\setcounter{enumi}{1}
\item If there exist $0\leq i_0\neq i_1\leq l$ such that $a_{i_0}\geq0$ and $a_{i_1}<0$ hold, then $X_c$ is contractible.
\item If there exist $0\leq i'_0\neq i'_1\leq l'$ such that $a'_{i'_0}\geq0$ and $a'_{i'_1}<0$ hold, then $X_c$ is contractible.
\item If $a\in\mathbb{Z}_{<0}^{l+1}$ and $a'\in\mathbb{Z}_{<0}^{l'+1}$ hold, then $X_c=\emptyset$
\item If $a\in\mathbb{Z}_{\ge0}^{l+1}$ and $a'\in\mathbb{Z}_{\ge0}^{l'+1}$ hold, then $X_c$ is contractible.
\item If $a\in\mathbb{Z}_{\ge0}^{l+1}$ and $a'\in\mathbb{Z}_{<0}^{l'+1}$ hold, then $X_c$ is homeomorphic to the $(l-1)$-dimensional sphere $S^{l-1}$.
\item If $a\in\mathbb{Z}_{<0}^{l+1}$ and $a'\in\mathbb{Z}_{\ge0}^{l'+1}$ hold, then $X_c$ is homeomorphic to the $(l'-1)$-dimensional sphere $S^{l'-1}$.
\end{enumerate}
\end{Lem}
\begin{proof}
(1) This follows by Lemma \ref{torscont}.

(2) Take $I\in\Delta_c$ with $i_0+1\notin I$. Then there exists $s_j\in\mathbb{Z}$ for $1\leq j\leq d+2$ such that $s_j\left\{
\begin{array}{ll}
=0 & (j\in I)\\
>0 & (j\notin I)
\end{array}
\right.$ and $\sum_{j=1}^{d+2}s_j\vec{z}_j=0$ hold. Take $u,v\in\mathbb{Z}_{>0}$ with $u\vec{x}_{i_0}=v\vec{x}_{i_1}$. Such $u$ and $v$ exist since $\pi(\vec{x}_{i_0}),\pi(\vec{x}_{i_1})>0$. For $1\le j\le d+2$, put 
\[t_j:=\left\{
\begin{array}{ll}
us_j & (j\neq i_0+1,i_1+1)\\
0 & (j=i_0+1)\\
vs_{i_0+1}+us_{i_1+1} & (j=i_1+1)
\end{array}
\right..\]
Then we have $t_j\left\{
\begin{array}{ll}
=0 & (j\in I\sqcup{\{i_0+1\}})\\
>0 & (j\notin I\sqcup{\{i_0+1\}})
\end{array}
\right.$ and $\sum_{j=1}^{d+2}t_j\vec{z}_j=0$. Thus $I\sqcup{\{i_0+1\}}\in\Delta_c$ holds. Therefore $X_c$ is homotopic to the point $\{v_{i_0+1}\}$.

(3) This can be shown in the same way as (2).

(4) This is obvious.

(5) Take $u\geq1$ so that $u\vec{y}_{i''}=0$ holds for $1\leq i''\leq d-l-l'$. Then by $\sum_{i''=1}^{d-l-l'}u\vec{y}_{i''}=0$, we can see $\{1,\cdots, l+l'+2\}\in\Delta_c$. Thus $X_c=\Conv\{v_1,\cdots,v_{l+l'+2}\}$ holds and this is contractible.

(6) Observe that $I:=\{1,\cdots,l+1\}\notin\Delta_c$ holds since for every $(s_j)_{j=1}^{d+2}\in\mathbb{Z}^{d+2}$ with $s_j\left\{
\begin{array}{ll}
=0 & (1\leq j\leq l+1)\\
>0 & (l+2\leq j\leq d+2)
\end{array}
\right.$, we have $\pi(\sum_{j=1}^{d+2}s_j\vec{z}_j)<0$. Take a proper subset $I'\subsetneq I$. Then by considering the signs of $\pi(\vec{z}_j)\in\mathbb{Z}$, there exists $(s_j)_{j=1}^{d+2}\in\mathbb{Z}^{d+2}$ with $s_j\left\{
\begin{array}{ll}
=0 & (j\in I')\\
>0 & (j\notin I')
\end{array}
\right.$ and $\sum_{j=1}^{d+2}s_j\vec{z}_j=0$. Thus we have $I'\in\Delta_c$. Therefore $X_c$ is homeomorphic to the $(l-1)$-dimensional sphere $S^{l-1}$.

(7) This can be shown in the same way as (6).
\end{proof}

\subsection{NCCRs}

Recall that the quotient group
$H=G/(\sum_{i''=1}^{d-l-l'}\mathbb{Z}\vec{y}_{i''})$ carries a partial
order and a $\mathbb{Z}$-action given by $n\cdot h=h+ns$. This action
satisfies (A1), (A2) and (A3).

\begin{Thm}\label{upNCCRcorr}
Let $R$ be a Gorenstein toric singularity with $\rk\Cl(R)=1$.
The assignment
\[
I\longmapsto [M_{q^{-1}(J(I))}]
\]
is a bijection from $\I_H$ to the isomorphism classes of basic
reflexive $R$-modules giving toric NCCRs of $R$.
\end{Thm}

To prove Theorem \ref{upNCCRcorr}, we first characterize when a direct
sum of divisorial modules is modifying.

\begin{Prop}\label{charmod}
For a subset $J\subseteq G$, the following conditions are equivalent.
\begin{enumerate}
\item $\bigoplus_{\vec{g}\in J}S_{\vec{g}}\in\refl R$ is a modifying module.
\item $q(J)\in\widetilde{\J}_H$
\item There exists $I\in\I_H$ such that $q(J)\subseteq J(I)$ holds.
\end{enumerate}
\end{Prop}
\begin{proof}
Recall that for $\vec{g},\vec{h}\in G$, we have $\Hom_R(S_{\vec{g}},S_{\vec{h}})\cong S_{\vec{h}-\vec{g}}$ by Proposition \ref{refhom}. Thus (1)$\Leftrightarrow$(2) follows from Theorem \ref{divCM}. (2)$\Leftrightarrow$(3) follows from Proposition \ref{charJX}.
\end{proof}

\begin{proof}[Proof of Theorem \ref{upNCCRcorr}]
By Proposition~\ref{charmaxJX}, $J(I)$ is a set of representatives
for the finite quotient $H/\mathbb{Z}s$. Since $\ker q$ is finite,
$q^{-1}(J(I))$ is nonempty and finite. The corresponding module is
basic. Its isomorphism class determines $q^{-1}(J(I))$ by
\eqref{addDegrees}, and hence determines $I$ by surjectivity of $q$
and Theorem~\ref{upJX}.

Let $M$ be a basic reflexive $R$-module giving a toric NCCR, and let
$J$ be the set of classes of its divisorial summands. Then
$M\cong M_J$. By Proposition~\ref{charmod},
$J\subseteq q^{-1}(J(I))$ for some $I\in\I_H$. For each
$\vec{h}\in q^{-1}(J(I))$, the module $M_J\oplus S_{\vec{h}}$ is
modifying. Hence $S_{\vec{h}}\in\add M_J$ by \cite[4.5]{IW}, and
\eqref{addDegrees} gives $J=q^{-1}(J(I))$.

Fix $I_0\in\I_H$ and set
\[
T=q^{-1}(J(I_0)),\qquad P=\bigoplus_{\vec{g}\in T}S(\vec{g}),\qquad
\Gamma=\End_R(P_0)\cong\End_S^G(P).
\]
Then $P_0$ is modifying by Proposition \ref{charmod}, and $F:=\Hom_S^G(P,-)$ is exact. Since $F(S(\vec{g}))$ is projective for $\vec{g}\in T$, we have
\[
J(I_0)\subseteq Y:=\{h\in H\mid \pd_\Gamma F(S(\vec{g}))<\infty
        \text{ for all }\vec{g}\in q^{-1}(h)\}.
\]

For a minimal element $m$ of $I\in\I_H$, put $I^-=\mu_m^-(I)$, so $J(I^-)=(J(I)\setminus\{m\})\cup\{m+s\}$. Set
\[
\begin{aligned}
A_+&=S/(x_0,\ldots,x_l),     &\qquad \vec{s}_+&=\sum_{i=0}^l\vec{x}_i,\\
A_-&=S/(x'_0,\ldots,x'_{l'}),&       \vec{s}_-&=-\sum_{j=0}^{l'}\vec{x}'_j.
\end{aligned}
\]
Both $\vec{s}_+$ and $\vec{s}_-$ lift $s$. The graded supports of $A_+$ and $A_-$ have images $\leq0$ and $\geq0$ under $q$, respectively. For $h\in J(I_0)$, the upper-set property gives $m+s-h\nleq0$ if $m\in I_0$ and $m-h\ngeq0$ otherwise.
Using $F(A(\vec{a}))\cong\bigoplus_{\vec{g}\in T}A_{\vec{a}-\vec{g}}$, we get the following exact sequences from the shifted Koszul resolutions, for $\vec{m}\in q^{-1}(m)$.
\[
\begin{aligned}
0&\to F(S(\vec{m}))\to\cdots\to F(S(\vec{m}+\vec{s}_+))
  \to F(A_+(\vec{m}+\vec{s}_+))=0 && (m\in I_0),\\
0&\to F(S(\vec{m}+\vec{s}_-))\to\cdots\to F(S(\vec{m}))
  \to F(A_-(\vec{m}))=0 && (m\notin I_0).
\end{aligned}
\]
Each intermediate shift $\vec{a}$ is $\vec{m}$ plus a nonempty proper subsum of $\vec{s}_+$ or $\vec{s}_-$. Thus $m<q(\vec{a})<m+s$, and minimality of $m$ gives $q(\vec{a})\in J(I)\setminus\{m\}$. If either window is contained in $Y$, all middle terms have finite projective dimension. With $\vec{m}\mapsto\vec{m}+\vec{s}_\pm$ identifying the two fibers, the exact sequences give
\[
J(I)\subseteq Y\quad\Longleftrightarrow\quad J(I^-)\subseteq Y.
\]
Corollary \ref{mutgen} now gives $Y=H$.

Lift a projective presentation of $X\in\mod\Gamma$ to a map $d$ in $\add P$ via $F\colon\add P\xrightarrow{\sim}\proj\Gamma$. Then $X\cong F(\Cok d)$. As $S$ is a polynomial ring, $\Cok d$ has a finite graded projective resolution. Applying $F$ gives $\pd_\Gamma X<\infty$, since $Y=H$ implies $\pd_\Gamma F(Q)<\infty$ for every $Q\in\proj^G\!S$.

The symmetric $R$-order $\Gamma$ \cite[2.10]{IW} is Gorenstein, so $\gl\Gamma\leq\dim R<\infty$ by the local Auslander--Buchsbaum formula \cite[2.16]{IW}.
\end{proof}

Retaining one copy of each isomorphism class of divisorial summands
preserves the additive closure and hence the NCCR property. Thus a
reflexive $R$-module $M$ gives a toric NCCR if and only if
\[
M\cong \bigoplus_{\vec{g}\in q^{-1}(J(I))}
           S_{\vec{g}}^{\oplus a_{\vec{g}}}
\]
for some $I\in\I_H$ and positive integers $a_{\vec{g}}$. Theorem
\ref{upNCCRcorr} determines $I$ uniquely, and the graded lift of $M$
determines the multiplicities by Proposition~\ref{degreeZeroEquivalence}.

Combining Theorem \ref{upNCCRcorr} and Proposition \ref{charmod}, we obtain the following corollary.

\begin{Cor}
Suppose there exists a subset $J\subseteq G$ such that $\bigoplus_{\vec{g}\in J}S_{\vec{g}}\in\refl R$ is a modifying module. Then there exists $J\subseteq J'\subseteq G$ such that $\bigoplus_{\vec{g}\in J'}S_{\vec{g}}\in\refl R$ gives a toric NCCR.
\end{Cor}

Van den Bergh's conjecture \cite[3.8]{VdB23} relates the number of
summand classes in an NCCR to the normalized volume of the associated
polytope. For a finite direct sum $M$ of divisorial modules, let
$|M|$ denote the number of distinct isomorphism classes of its
divisorial summands.

\begin{Thm}\label{volrk}
Let $R$ be a Gorenstein toric singularity with $\rk\Cl(R)=1$ and $P$
the corresponding $d$-dimensional lattice polytope. For every
reflexive module $M$ giving a toric NCCR of $R$, we have
\[|M|=d!\vol(P).\]
\end{Thm}
\begin{proof}
By the classification above, the set of classes of divisorial
summands of $M$ is $q^{-1}(J(I))$ for some $I\in\I_H$.
Proposition~\ref{charmaxJX} and Theorem~\ref{vol} give
\[
|M|=\#q^{-1}(J(I))=|\ker q|\,|H/\mathbb Zs|=d!\vol(P).
\]
\end{proof}

\section{Iyama--Wemyss mutations of toric NCCRs}

Removing a minimal element of an upper set replaces one representative in its window by its translate by $s$. The corresponding change of divisorial summands can be realized using the Koszul complex on either family of variables.

\begin{Thm}\label{toricmut}
Let $m$ be a minimal element of $I\in\I_H$ and put $I^-=\mu_m^-(I)$ and $M=M_{q^{-1}(J(I)\setminus\{m\})}$. With the Koszul approximations in the proof, we have
\[
(\mu_M^+)^{l'}(M_{q^{-1}(J(I))})\cong M_{q^{-1}(J(I^-))}
\cong(\mu_M^-)^l(M_{q^{-1}(J(I))}).
\]
\end{Thm}
\begin{proof}
Write $T=q^{-1}(J(I)\setminus\{m\})$. Since the image of $s$ in the free quotient of $H$ is at least $2$, $J(I)$ has at least two elements and $M=M_T\ne0$.
Let $F_\bullet$ and $F'_\bullet$ be the graded Koszul resolutions of $S/(x_0,\ldots,x_l)$ and $S/(x'_0,\ldots,x'_{l'})$. Put $K_0=K'_0=S$ and, for positive indices, let $K_i=\Ker(F_i\to F_{i-1})$ and $K'_j=\Ker(F'_j\to F'_{j-1})$. These kernels are reflexive.

For $\vec{m}\in q^{-1}(m)$, the maps
\[
(F_i)_{-\vec{m}}\longrightarrow(K_{i-1})_{-\vec{m}}
\quad (1\leq i\leq l),\qquad
(F'_j)_{\vec{m}}\longrightarrow(K'_{j-1})_{\vec{m}}
\quad (1\leq j\leq l')
\]
have sources in $\add M^*$ and $\add M$, respectively. Indeed, the degrees of the source summands, after dualizing in the first case, have images strictly between $m$ and $m+s$, hence in $J(I)\setminus\{m\}$.
For $\vec{g}\in T$, applying $\Hom_R(S_{-\vec{g}},-)$ and $\Hom_R(S_{\vec{g}},-)$ to the respective maps takes the degree $\vec{g}-\vec{m}$ and $\vec{m}-\vec{g}$ components of the graded maps by Proposition \ref{degreeZeroEquivalence}. Minimality of $m$ gives $q(\vec{g}-\vec{m})\nleq0$ and $q(\vec{m}-\vec{g})\ngeq0$. The respective Koszul quotients vanish in these degrees, so the maps are right approximations.

Sum the maps over $\vec{m}\in q^{-1}(m)$ and adjoin the identity on $M^*$ or $M$. Their kernels are the corresponding sums of the $K_i$ or $K'_j$ components, giving the next mutation in each sequence. The final summands replacing the ones over $m$ are
\[
((K_l)_{-\vec{m}})^*\cong S_{\vec{m}+\sum_i\vec{x}_i},\qquad
(K'_{l'})_{\vec{m}}\cong S_{\vec{m}-\sum_j\vec{x}'_j}.
\]
In each case the new degrees run through $q^{-1}(m+s)$, giving $M_{q^{-1}(J(I^-))}$.
\end{proof}

As an immediate corollary, we obtain the following.

\begin{Cor}\label{toricNCCRconn}
Any two basic reflexive $R$-modules giving toric NCCRs of $R$ can be
joined by a finite sequence of Iyama--Wemyss mutations. In
particular, all toric NCCRs of $R$ are derived equivalent.
\end{Cor}
\begin{proof}
By Theorem~\ref{upNCCRcorr}, the two modules correspond to upper
sets $I_1,I_2\in\I_H$. By Proposition~\ref{muttrans} and
Theorem~\ref{toricmut}, right mutations take each module to
$M_{q^{-1}(J(I_1\cap I_2))}$.

Each right mutation in the second sequence uses an approximation
$\mathrm{id}_N\oplus f\colon N\oplus E\to N\oplus X$ at the common
summand $N$. Write $K=\Ker f$ and let $i\colon K\hookrightarrow E$
be the inclusion. By \cite[Proposition~6.4(3)]{IW}, the dual
sequence
\[
0\to N^*\oplus X^*\xrightarrow{\mathrm{id}_{N^*}\oplus f^*}
N^*\oplus E^*\xrightarrow{(0,i^*)}K^*
\]
is exact and $(0,i^*)$ is a right $(\add N^*)$-approximation.
Hence $\mathrm{id}_{N^*}\oplus i^*$ is a right
$(\add N^*)$-approximation of $N^*\oplus K^*$ with kernel
isomorphic to $X^*$. The resulting left mutation at $N$ takes
$N\oplus K$ back to $N\oplus X$ up to isomorphism. Reversing the
second sequence joins the two modules.

The endomorphism algebras along the resulting sequence are derived
equivalent by \cite[6.8,6.10]{IW}. Removing repeated divisorial
summands preserves the Morita equivalence class of the endomorphism
algebra, so the derived equivalence holds for all toric NCCRs.
\end{proof}

\section{Periodic cut quivers for toric NCCRs}

We work in the non-pyramidal case $d=l+l'$ (Proposition~\ref{pyra}),
so the associated polytope $P$ is simplicial. We use cuts of a
periodic quiver on the $d$-torus to present the endomorphism algebras
of the basic reflexive modules in Theorem~\ref{upNCCRcorr}.

\subsection{Cut detectors}

Let $l,l'\geq1$. Let $e_i\in\mathbb{Z}^{l+1}$ be the $i$-th unit vector for $0\leq i\leq l$. Put $\alpha_i:=e_i-e_{i-1}$ for $1\leq i\leq l$ and $\alpha_0:=e_0-e_l$. Let $L:=\{v=(v_i)_{i=0}^l\in\mathbb{Z}^{l+1}\mid\sum_{i=0}^lv_i=0\}=\sum_{i=0}^l\mathbb{Z}\alpha_i\subseteq\mathbb{Z}^{l+1}$ be a $l$-dimensional lattice. Similarly, we define a $l'$-dimensional lattice $L'=\sum_{i'=0}^{l'}\mathbb{Z}\alpha'_{i'}\subseteq\mathbb{Z}^{l'+1}$. We define an infinite quiver $\widehat{Q}$ as
\[\widehat{Q}_0:=L\oplus L' \text{ and}\]
\[\widehat{Q}_1:=\bigsqcup_{i=0}^l\{x\to x+\alpha_i\mid x\in L\oplus L'\}\sqcup\bigsqcup_{i'=0}^{l'}\{x\to x+\alpha'_{i'}\mid x\in L\oplus L'\}.\]
We say that an arrow $x\to x+\alpha_i$ in $\widehat{Q}$ has {\it type} $\alpha_i$ and an arrow $x\to x+\alpha'_{i'}$ has {\it type} $\alpha'_{i'}$. A cycle of length $l+l'+2$ in $\widehat{Q}$ consisting of arrows of $l+l'+2$ distinct types is called an {\it elementary cycle}.

The condition on elementary cycles will be used to compare the weights of paths that take successive steps from different families in opposite orders.

\begin{Def}\label{defcut}
 A subset $\widehat{C}\subseteq\widehat{Q}_1$ is called a {\it cut} if the following conditions are satisfied.
 \begin{enumerate}
 \item Every cycle of length $l+1$ in $\widehat{Q}$ consisting of arrows of type $\alpha_0,\cdots,\alpha_l$ has exactly one arrow in $\widehat{C}$.
 \item Every cycle of length $l'+1$ in $\widehat{Q}$ consisting of arrows of type $\alpha'_0,\cdots,\alpha'_{l'}$ has exactly one arrow in $\widehat{C}$.
 \item Every elementary cycle in $\widehat{Q}$ has the same number of arrows in $\widehat{C}$ of type $\alpha$ and arrows in $\widehat{C}$ of type $\alpha'$.
 \end{enumerate}
 \end{Def}

Take a cofinite subgroup $B\subseteq L\oplus L'$ and put $m:=|(L\oplus L')/B|$. Define a quiver $Q$ as
\[Q_0:=(L\oplus L')/B\text{ and}\]
\[Q_1:=\bigsqcup_{i=0}^l\{x+B\to x+\alpha_i+B\mid x\in L\oplus L'\}\sqcup\bigsqcup_{i'=0}^{l'}\{x+B\to x+\alpha'_{i'}+B\mid x\in L\oplus L'\}.\] $Q$ is a finite quiver which may have multiple arrows.  A cut $\widehat{C}\subseteq\widehat{Q}_1$ is said to be $B$-{\it periodic} if $\widehat{C}$ is invariant under $B$-translation. Then a $B$-periodic cut $\widehat{C}\subseteq\widehat{Q}_1$ induces a subset $C\subseteq Q_1$ naturally, which we call a {\it cut} of $Q$. For a cut $C\subseteq Q_1$, we put
\[\gamma(C):=(\sharp\{a\in C\mid\text{The type of $a$ is }\alpha_i.\})_{i=0}^l\in\mathbb{Z}^{l+1}_{\geq0}.\]
We also define $\gamma'(C)\in\mathbb{Z}^{l'+1}_{\geq0}$ similarly. We call the pair $(\gamma(C),\gamma'(C))$ the {\it type} of $C$. For a cut $C\subseteq Q_1$ of type $((\gamma_i)_{i=0}^l,(\gamma'_i)_{i=0}^{l'})$, we have $\sum_{i=0}^l\gamma_i=\sum_{i=0}^{l'}\gamma'_i=m$ by an easy counting argument. 

For a fixed type, we encode membership in a cut by the increments of an integer-valued function on the vertices of $Q$.

\begin{Def}
A map $f\colon Q_0\to\mathbb{Z}$ is a {\it cut detector} of type $(\gamma,\gamma')\in\mathbb{Z}^{l+1}_{\geq0}\times\mathbb{Z}^{l'+1}_{\geq0}$ if $f(0)=0$ and, for every $x\in L\oplus L'$,
\begin{align*}
f(x+\alpha_i+B)-f(x+B)&\in\{\gamma_i,\gamma_i-m\} && (0\leq i\leq l),\\
f(x+\alpha'_j+B)-f(x+B)&\in\{-\gamma'_j,-\gamma'_j+m\} && (0\leq j\leq l').
\end{align*}
\end{Def}

Prescribing such vertex values gives a way to construct cuts in the cut-upper set correspondence.

\begin{Thm}\label{cutbij}
Fix $(\gamma,\gamma')\in\mathbb{Z}^{l+1}_{\geq0}\times\mathbb{Z}^{l'+1}_{\geq0}$ satisfying
\[
\sum_{i=0}^l\gamma_i=\sum_{j=0}^{l'}\gamma'_j=m.
\]
There is a bijection between the following sets.
\begin{enumerate}
\item The set of cut detectors $f\colon Q_0\to\mathbb{Z}$ of type $(\gamma,\gamma')$.
\item The set of cuts $C\subseteq Q_1$ of type $(\gamma,\gamma')$.
\end{enumerate}
\end{Thm}

To construct the map from (2) to (1), we first prescribe the detector's increments as weights on the arrows of $\widehat{Q}$.

\begin{Def}
Let $C$ be a cut of $Q$ of type $(\gamma,\gamma')$ and $\widehat{C}$ its periodic lift. For $a\in\widehat{Q}_1$ of type $\alpha_i$, we define
\[f_C(a):=
  \begin{cases}
    \gamma_i & a\notin\widehat{C},\\
    \gamma_i-m & a\in\widehat{C}.
  \end{cases}
\]
For $a\in\widehat{Q}_1$ of type $\alpha'_i$, we define
\[f_C(a):=
  \begin{cases}
    -\gamma'_i & a\notin\widehat{C},\\
    -\gamma'_i+m & a\in\widehat{C}.
  \end{cases}
\]
For a path $p=a_n\cdots a_1$ in $\widehat{Q}$, we define
\[f_C(p):=\sum_{i=1}^nf_C(a_i).\]
\end{Def}

\begin{Rem}
For a path $p$ in $\widehat{Q}$ of length $0$, we think $f_C(p)=0$.
\end{Rem}

To obtain a function on vertices, we need the weight of a path from $0$ to depend only on its endpoint.

\begin{Lem}
For paths $p,q$ in $\widehat{Q}$ with the same source and target, we have $f_C(p)=f_C(q)$.
\end{Lem}
\begin{proof}
The type sums and the cut conditions give $f_C=0$ on the cycles in Definition \ref{defcut}. As in \cite[2.5]{DG}, we compare paths by swapping consecutive types and inserting cycles. For a swap, complete both orderings by the same tail, using the remaining types in that family if the two types are in one family, or in both families otherwise. Each resulting closed path is a cycle in Definition \ref{defcut} or decomposes into two single-family cycles. Thus both completions have weight zero and the two orderings have equal weight. Within each family, the relations among the arrow vectors are generated by the sum of the vectors being zero. Hence the type multiplicities in $p$ and $q$ differ by a constant within each family. Inserting single-family cycles makes the multiplicities agree, and successive swaps then identify the paths.
\end{proof}

The arrow vectors generate $L\oplus L'$ as a monoid, since the negative of each vector is the sum of the others in its family. Thus for $x\in L\oplus L'$ we can define
\[f_C(x):=f_C(p_x),\]
where $p_x$ is any path from $0$ to $x$. The lemma shows that this is well-defined.

To obtain a detector on the finite quiver $Q$, it remains to show that these vertex values are invariant under translation by $B$.

\begin{Prop}
The function $f_C\colon L\oplus L'\to\mathbb{Z}$ descends to a cut detector $f_C\colon Q_0\to\mathbb{Z}$ of type $(\gamma,\gamma')$.
\end{Prop}
\begin{proof}
For $x\in L\oplus L'$ and $b\in B$, translating a path from $0$ to $x$ by $b$ gives $f_C(x+b)=f_C(x)+f_C(b)$, by periodicity. In particular, $f_C|_B$ is a group homomorphism.

Choose a lift $x_v\in L\oplus L'$ of each $v\in Q_0$. For an arrow type $\alpha$ in either family, put $\bar\alpha=\alpha+B$. The periodic arrow weights of the $m$ arrows of type $\alpha$ in $Q$ sum to zero by the definition of the cut type. Since translation by $\bar\alpha$ permutes $Q_0$, we obtain
\[
0=\sum_{v\in Q_0}\bigl(f_C(x_v+\alpha)-f_C(x_v)\bigr)
 =\sum_{v\in Q_0}f_C(x_v+\alpha-x_{v+\bar\alpha})
 = f_C(m\alpha).
\]
Here each $x_v+\alpha-x_{v+\bar\alpha}$ belongs to $B$. The elements $m\alpha$ generate $m(L\oplus L')$, so $mf_C(b)=f_C(mb)=0$ for all $b\in B$. Thus $f_C$ descends to $Q_0$, with $f_C(0)=0$ and the prescribed increments.
\end{proof}

Modulo $m$, the detector's increments depend only on the arrow types. Periodicity therefore imposes the following congruence on the cut type, as in \cite[Proposition~2.15]{DG}.

\begin{Cor}
Let $C$ be a cut of $Q$ and $(\gamma,\gamma')$ its type. Take $((m_i)_{i=0}^l,(m'_i)_{i=0}^{l'})\in\mathbb{Z}^{l+1}\times\mathbb{Z}^{l'+1}$. If $\sum_{i=0}^lm_i\alpha_i+\sum_{i=0}^{l'}m'_i\alpha'_i\in B$ holds, then we have $\sum_{i=0}^lm_i\gamma_i-\sum_{i=0}^{l'}m'_i\gamma'_i\in m\mathbb{Z}$.
\end{Cor}
\begin{proof}
By the definition of $f_C$, we have $f_C(\sum_{i=0}^lm_i\alpha_i+\sum_{i=0}^{l'}m'_i\alpha'_i)-(\sum_{i=0}^lm_i\gamma_i-\sum_{i=0}^{l'}m'_i\gamma'_i)\in m\mathbb{Z}$. Since $f_C(\sum_{i=0}^lm_i\alpha_i+\sum_{i=0}^{l'}m'_i\alpha'_i)=0$, we get the conclusion.
\end{proof}

With $f_C$ constructed, it remains to check that the arrows selected by a detector form a cut of the prescribed type.

\begin{proof}[Proof of Theorem \ref{cutbij}]
Given a cut detector $f$ of type $(\gamma,\gamma')$, define $C_f\subseteq Q_1$ by the increments along each arrow $a\colon x\to y$.
\[a\in C_f\Leftrightarrow
  \begin{cases}
    f(y)-f(x)=\gamma_i-m & (a\text{ is of type }\alpha_i),\\
    f(y)-f(x)=-\gamma'_j+m & (a\text{ is of type }\alpha'_j).
\end{cases} \]
For a cycle in Definition \ref{defcut}, let $r$ and $r'$ count its unprimed and primed arrows in the lift of $C_f$. The sums of the increments on the three kinds of cycles are $m(1-r)$, $m(r'-1)$ and $m(r'-r)$, respectively. Each sum vanishes by telescoping, so $C_f$ is a cut.

For a fixed type $\alpha_i$, the sources and targets of its arrows each run through $Q_0$ once. Summing the increments therefore gives $0=m\gamma_i-mr_i$, where $r_i$ counts the arrows of type $\alpha_i$ in $C_f$. Thus $r_i=\gamma_i$, and the same calculation gives $\gamma'_j$ for type $\alpha'_j$.

The increments of $f_C$ recover membership in $C$, so $C_{f_C}=C$. The functions $f_{C_f}$ and $f$ have the same increments and both vanish at $0$, so they agree on the connected quiver $Q$.
\end{proof}

For fixed $B$, the existence of a periodic cut can be decided from the type sums and the congruence above.

\begin{Thm}\label{chartype}
For $(\gamma,\gamma')=((\gamma_i)_{i=0}^l,(\gamma'_{i'})_{i'=0}^{l'})\in\mathbb{Z}_{\geq0}^{l+1}\times\mathbb{Z}_{\geq0}^{l'+1}$, $(\gamma,\gamma')$ is a type of a $B$-periodic cut if and only if both of the following conditions are satisfied.
\begin{enumerate}
\item $\sum_{i=0}^l\gamma_i=\sum_{i'=0}^{l'}\gamma'_{i'}=m$
\item For any $((m_i)_{i=0}^l,(m'_{i'})_{i'=0}^{l'})\in\mathbb{Z}^{l+1}\times\mathbb{Z}^{l'+1}$ with $\sum_{i=0}^lm_i\alpha_i+\sum_{i'=0}^{l'}m'_{i'}\alpha'_{i'}\in B$, we have $\sum_{i=0}^lm_i\gamma_i-\sum_{i'=0}^{l'}m'_{i'}\gamma'_{i'}\in m\mathbb{Z}$.
\end{enumerate}
\end{Thm}

The counting identities and the congruence above give necessity. In the next subsection, we prove sufficiency by constructing a cut from a set of coset representatives.

\subsection{Cut-upper set correspondence}

Let $(\gamma,\gamma')=((\gamma_i)_{i=0}^l,(\gamma'_{i'})_{i'=0}^{l'})\in\mathbb{Z}_{\geq0}^{l+1}\times\mathbb{Z}_{\geq0}^{l'+1}$ be a pair satisfying the conditions (1) and (2) in Theorem \ref{chartype}. We define group homomorphisms $\Phi\colon\mathbb{Z}^{l+1}\to\mathbb{Z}\oplus((L\oplus L')/B)$ and $\Phi'\colon\mathbb{Z}^{l'+1}\to\mathbb{Z}\oplus((L\oplus L')/B)$ by
\[\Phi(e_i):=(\gamma_i,\alpha_i+B)\text{ and }\Phi'(e_{i'}):=(-\gamma'_{i'},\alpha'_{i'}+B).\]
Put $G=G(B,\gamma,\gamma'):=\Im\Phi+\Im\Phi'\subseteq\mathbb{Z}\oplus((L\oplus L')/B)$ and $\vec{x}_i:=\Phi(e_i),\vec{x}_{i'}':=\Phi'(e_{i'})\in G$. Observe that $\vec{s}:=\sum_{i=0}^l\vec{x}_i=-\sum_{i'=0}^{l'}\vec{x}_{i'}'=(m,0)$ holds. Thus the composition $G\hookrightarrow\mathbb{Z}\oplus((L\oplus L')/B)\twoheadrightarrow(L\oplus L')/B$ induces a group homomorphism $\phi\colon G/\mathbb{Z}\vec{s}\to(L\oplus L')/B$.

To encode a cut by representatives in $G$, we first identify each vertex of $Q$ with a coset modulo $\mathbb{Z}\vec{s}$.

\begin{Lem}
The group homomorphism $\phi\colon G/\mathbb{Z}\vec{s}\to (L\oplus L')/B$ is an isomorphism.
\end{Lem}
\begin{proof}
The surjectivity follows from $\phi(\vec{x}_i+\mathbb{Z}\vec{s})=\alpha_i+B$ and $\phi(\vec{x}_{i'}'+\mathbb{Z}\vec{s})=\alpha'_{i'}+B$. Take $\vec{g}=\Phi(v)+\Phi'(v')\in G$ with $\phi(\vec{g}+\mathbb{Z}\vec{s})=0$. If we put $v=(m_i)_{i=0}^l$ and $v'=(m'_{i'})_{i'=0}^{l'}$, then we have $\sum_{i=0}^lm_i\alpha_i+\sum_{i'=0}^{l'}m'_{i'}\alpha'_{i'}\in B$. Thus by our assumption, there exists $n\in\mathbb{Z}$ with $\sum_{i=0}^lm_i\gamma_i-\sum_{i'=0}^{l'}m'_{i'}\gamma'_{i'}=mn$. This implies $\vec{g}=n\vec{s}$.
\end{proof}

We define
\begin{align*}
\J:=\{&J\subseteq G\colon\text{a complete set of representatives for }G/\mathbb{Z}\vec{s}\mid \\
&\vec{g}+\vec{x}_i,\vec{g}-\vec{x}_{i'}'\in J\sqcup(J+\vec{s})\text{ for all }\vec{g}\in J,0\leq i\leq l\text{ and }0\leq i'\leq l'\}.
\end{align*}
Let $\pi:=(G\hookrightarrow\mathbb{Z}\oplus(L\oplus L')/B\twoheadrightarrow\mathbb{Z})$ denote the composition of natural group homomorphisms. The condition on the representatives ensures that their normalized $\pi$-values define a cut detector. This reduces existence in Theorem \ref{chartype} to finding one set of representatives in $\J$.

\begin{Prop}\label{cutJcorr}
We have a surjective map
\[C(-)\colon\J\to\{\text{Cuts of $Q$ of type }(\gamma,\gamma')\}.\]
For $J,J'\in\J$, $C(J)=C(J')$ holds if and only if $J=J'+n\vec{s}$ holds for some $n\in\mathbb{Z}$.
\end{Prop}
\begin{proof}
Use $\phi$ to identify $Q_0$ with $G/\mathbb{Z}\vec{s}$. For $J\in\J$, let $\vec{g}_v\in J$ represent $v$ and put $f_J(v)=\pi(\vec{g}_v-\vec{g}_0)$. For an arrow $a\colon v\to w$ of degree $\vec{z}=\vec{x}_i$ or $\vec{x}'_j$, the condition on $J$ gives
\[
\vec{g}_v+\vec{z}=\vec{g}_w+\varepsilon\vec{s},
\qquad
\varepsilon\in
\begin{cases}
\{0,1\} & (\vec{z}=\vec{x}_i),\\
\{0,-1\} & (\vec{z}=\vec{x}'_j).
\end{cases}
\]
For the primed family, apply the condition to $\vec{g}_w-\vec{x}'_j$. Thus $f_J(w)-f_J(v)=\pi(\vec{z})-m\varepsilon$, so $f_J$ is a cut detector of the prescribed type. The cut $C(J):=C_{f_J}$ consists of the arrows for which $\varepsilon\ne0$.

Conversely, let $f$ be a cut detector of this type. Modulo $m$, the functions $v\mapsto f(v)$ and $v\mapsto\pi(\vec{g})$ for $\vec{g}\in v$ agree at $0$ and have the same increments. Connectedness gives $f(v)\equiv\pi(\vec{g})\pmod m$, so each $v$ has a unique representative $\vec{g}_v$ with $\pi(\vec{g}_v)=f(v)$. The increment conditions give the displayed conditions on $\varepsilon$, so these representatives form a set $J\in\J$ with $f_J=f$.

This proves surjectivity by Theorem \ref{cutbij}, which also gives $C(J)=C(J')$ exactly when $f_J=f_{J'}$. By uniqueness of the normalized representatives, this is equivalent to $J$ and $J'$ differing by a translate by a multiple of $\vec{s}$.
\end{proof}

Now we can prove Theorem \ref{chartype}.

\begin{proof}[Proof of Theorem \ref{chartype}]
The set $J=\{\vec{g}\in G\mid0\leq\pi(\vec{g})<m\}$ has one representative per coset. Each shift in the definition of $\J$ adds a first coordinate in $[0,m]$, so $J\in\J$. Proposition \ref{cutJcorr} gives a cut of the prescribed type.
\end{proof}

When every entry of $(\gamma,\gamma')$ is positive, $G$ and its two families of generators satisfy (G1), (G2) and (G3). The sets in $\J$ are then windows of non-trivial upper sets, which links cuts to the classification in Theorem \ref{upNCCRcorr}.

\begin{Thm}\label{cutupcorr}(Cut-upper set correspondence)
Assume $(\gamma,\gamma')\in\mathbb{Z}_{>0}^{l+1}\times\mathbb{Z}_{>0}^{l'+1}$. Then we have a surjective map
\[C(-)\colon\I_G\to\{\text{Cuts of $Q$ of type }(\gamma,\gamma')\}.\]
For $I,I'\in\I_G$, $C(I)=C(I')$ holds if and only if $I=I'+n\vec{s}$ for some $n\in\mathbb{Z}$.
\end{Thm}
\begin{proof}
Each positive generator $\vec{u}=\vec{x}_i$ or $-\vec{x}'_j$ satisfies $0\leq\vec{u}\leq\vec{s}$. For $\vec{g}\in J(I)$, we have $\vec{g}+\vec{u}\in I$ and $\vec{g}+\vec{u}-2\vec{s}\leq\vec{g}-\vec{s}\notin I$. Hence
\[
\vec{g}+\vec{u}\in I\setminus(I+2\vec{s})=J(I)\sqcup(J(I)+\vec{s}).
\]
Conversely, for $J\in\J$, the set $I=J+\mathbb{Z}_{\geq0}\vec{s}$ is a non-trivial upper set by the defining condition on $J$, and $J(I)=J$. By Lemma \ref{boundary}, these assignments are inverse and commute with translation. The result now follows from Proposition \ref{cutJcorr}.
\end{proof}

To apply the cut construction to a toric ring, we need to recover a period lattice and a cut type from its grading data. The same data determine the corresponding non-pyramidal lattice polytope up to affine lattice equivalence.

\begin{Prop}\label{3corr}
\begin{enumerate}
\item For integers $l,l'\ge1$, we have a bijection between the following sets.
\begin{enumerate}
\item Cofinite subgroups $B\subseteq L\oplus L'$ with pairs $(\gamma,\gamma')\in\mathbb{Z}_{>0}^{l+1}\times\mathbb{Z}_{>0}^{l'+1}$ satisfying the conditions in Theorem \ref{chartype}.
\item Finitely generated rank-one abelian groups $G$ with ordered families $(\vec{x}_i)_{i=0}^l$ and $(\vec{x}'_j)_{j=0}^{l'}$ satisfying (G1), (G2) and (G3), up to group isomorphisms preserving the ordered families.
\end{enumerate}
\item Let $X_{l,l'}$ denote the set in (1)(b). For $d\ge2$, we have a bijection between the following sets.
\begin{enumerate}
\item The quotient of $\bigsqcup_{l+l'=d}X_{l,l'}$ by permutations within each family and exchange of the two families.
\item The $d$-dimensional lattice polytopes with $d+2$ vertices that are not pyramids, up to affine lattice equivalence.
\end{enumerate}
\end{enumerate}
\end{Prop}
\begin{proof}
(1) The map from (a) to (b) is the construction of $G(B,\gamma,\gamma')$. Conversely, for data in (b), put $\vec{s}=\sum_i\vec{x}_i=-\sum_j\vec{x}'_j$ and let $\pi_0\colon G\to\mathbb{Z}$ be the projection to the free quotient with $\pi_0(\vec{x}_i)>0$. Set $t=|G_{\rm tors}|$ and $m=|G/\mathbb{Z}\vec{s}|=t\pi_0(\vec{s})$. Take $B$ to be the kernel of the surjection $L\oplus L'\to G/\mathbb{Z}\vec{s}$ sending $\alpha_i,\alpha'_j$ to the classes of $\vec{x}_i,\vec{x}'_j$. Then
\[
\gamma_i=t\pi_0(\vec{x}_i),\qquad
\gamma'_j=-t\pi_0(\vec{x}'_j).
\]
Both sums equal $m$, and the congruence in Theorem \ref{chartype} follows by applying $t\pi_0$ to a combination of the weights in $\mathbb{Z}\vec{s}$. The embedding $\vec{g}\mapsto(t\pi_0(\vec{g}),\vec{g}+\mathbb{Z}\vec{s})$ identifies $G$ with $G(B,\gamma,\gamma')$. For a group constructed from (a), the quotient by $\mathbb{Z}\vec{s}$ is $Q_0$ via $\phi$ and $\pi(\vec{s})=m$. Thus $\pi=t\pi_0$, and the inverse recovers $B$ and both types.

(2) By (G1) and (G2), deleting any weight leaves a generating family. By (G3) and $l,l'\ge1$, the remaining weights have a positive relation. Section~\ref{toricGradingData} therefore gives the lattice polytope with the indicated ordered vertices. It is non-pyramidal by Proposition \ref{pyra}. Conversely, the weights of such a polytope lie in a rank-one group and all have non-zero free part, so choosing a sign for the free quotient and ordering the two families gives data in (1)(b). Forgetting these choices gives exactly the quotient in (2)(a).
\end{proof}

\subsection{Quivers with relations of toric NCCRs}

Let $C$ be a cut of $Q$ of positive type and let $\widehat{C}$ be its periodic lift. Keep the vertices of $\widehat{Q}$ and the arrows outside $\widehat{C}$. For each $x\in L\oplus L'$ and each pair $(i,j)$ for which both $x\to x+\alpha_i$ and $x\to x+\alpha'_j$ belong to $\widehat{C}$, add a mixed arrow of type $(i,j)$ from $x$ to $x+\alpha_i+\alpha'_j$. This gives a $B$-periodic quiver $\widehat{Q}(\widehat{C})$. Put $Q(C)=\widehat{Q}(\widehat{C})/B$.

Label retained arrows of types $\alpha_i$ and $\alpha'_j$ by $x_i$ and $x'_j$, respectively, and mixed arrows of type $(i,j)$ by $x_ix'_j$. Write $\rho(p)$ for the product of the labels along a path $p$, with $\rho(e_v)=1$ for the path of length zero at $v$. Let $I(C)$ be the ideal of $kQ(C)$ generated by $p-p'$ for paths with the same source, target and label, and put
\[
\Gamma(B,C):=kQ(C)/I(C).
\]
At each source the arrows have distinct labels. A mixed arrow's label cannot occur on a longer path from that source, since both of its single-variable factors belong to cut arrows there. Thus $I(C)\subseteq kQ(C)_{\geq2}$.

For an upper set $I$ corresponding to $C$, the vertices $J(I)$
index the summands of $M_{J(I)}$ in Theorem~\ref{upNCCRcorr}.
The relations defining $I(C)$ have the form used for quivers of
sections \cite[\S~2.2]{CQV12}. The presentation follows once
compositions of the prescribed arrows realize every monomial
multiplication map between the summands.

\begin{Thm}[Quiver presentation from a cut]\label{dimerreal}
Let $G$ and its two ordered families of weights satisfy (G1), (G2)
and (G3) with $d=l+l'$. Grade
$S=k[x_0,\ldots,x_l,x'_0,\ldots,x'_{l'}]$ by these weights and put
$R=S_0$. Use Proposition \ref{3corr}(1) to obtain
$(B,\gamma,\gamma')$. For $I\in\I_G$, let $C=C(I)$ be the cut from
Theorem \ref{cutupcorr} and identify the vertices of $Q(C)$ with
$J:=J(I)$.
\begin{enumerate}
\item\label{dimerlabels}
For $\vec{g},\vec{h}\in J$ and a monomial $u\in S$, there is a path
$p\colon\vec{g}\to\vec{h}$ in $Q(C)$ with $\rho(p)=u$ if and only if
$\deg_G u=\vec{h}-\vec{g}$.
\item\label{dimeralgebra}
Sending each path to multiplication by its label induces an isomorphism
\[
\Gamma(B,C)\xrightarrow{\sim}\End_R(M_J).
\]
\end{enumerate}
\end{Thm}
\begin{proof}
Write $n_I(\vec{h})$ for the unique integer with
$\vec{h}\in J+n_I(\vec{h})\vec{s}$. By Proposition \ref{cutJcorr},
a step of unprimed type increases $n_I$ by $1$ if the arrow is cut
and by $0$ otherwise. A primed step decreases it by $1$ or $0$ in
the same cases. Thus a retained arrow starting in $J$ ends at the
representative obtained by adding its degree. If the two arrows
of types $\alpha_i,\alpha'_j$ out of $\vec{g}\in J$ are cut, taking
these steps in the two orders gives
\[
n_I(\vec{g}+\vec{x}_i+\vec{x}'_j)\in\{0,1\}\cap\{-1,0\}=\{0\}.
\]
Thus the same assertion holds for a mixed arrow, and the degree of
a path label is its target minus its source.

For the converse in (1), induct on the total degree of $u$.
The case $u=1$ is immediate. If a factor labels a retained arrow
out of $\vec{g}$, follow that arrow and apply induction to the
remaining factors. Otherwise every factor corresponds to a cut
arrow out of $\vec{g}$. If all factors were from one family, the
first step would change $n_I$ by $\pm1$ and all later steps would
have the same weak sign, contradicting $n_I(\vec{h})=0$. Choose a
factor from each family, follow the corresponding mixed arrow,
and apply induction.

For (2), Proposition \ref{refhom} identifies
$\Hom_R(S_{\vec{g}},S_{\vec{h}})$ with $S_{\vec{h}-\vec{g}}$ by
multiplication. By (1), the path labels give the monomial basis of
this component after identifying paths with equal labels. These
identifications respect composition since path labels multiply.
\end{proof}

Put $D:=\prod_{i=0}^l x_i\prod_{j=0}^{l'}x'_j$, the anticanonical
monomial \cite[\S~3.1]{CQV12}. A closed path with label $D$ in
$Q(C)$ or $\widehat{Q}(\widehat{C})$ is called an {\it elementary cycle}.

\begin{Cor}\label{dimercycles}
Every path in $Q(C)$ or $\widehat{Q}(\widehat{C})$ whose label divides
$D$ extends to an elementary cycle. In particular, every arrow
belongs to an elementary cycle.
\end{Cor}
\begin{proof}
If $p\colon\vec{g}\to\vec{h}$ is a path in $Q(C)$ with $\rho(p)\mid D$,
then $\deg_G(D/\rho(p))=\vec{g}-\vec{h}$ since $\deg_G D=0$.
Theorem \ref{dimerreal}(\ref{dimerlabels}) gives a return path with
label $D/\rho(p)$. In the cover, lift the return path from the end of
the given path. The total displacement is
$\sum_i\alpha_i+\sum_j\alpha'_j=0$, so the lift returns to the original
vertex.
\end{proof}

To compare labels with homotopy classes, realize $Q(C)$ on
$T^d=(L\oplus L')_{\mathbb{R}}/B$ by projecting the straight segments
joining the endpoints of the lifted arrows.

Let $\delta\colon\mathbb{Z}^{d+2}\twoheadrightarrow L\oplus L'$
send the basis vectors to the corresponding steps, and put
$M=\ker(\deg_G\colon\mathbb{Z}^{d+2}\to G)$. The kernel of $\delta$
consists of vectors constant on each family, with $G$-degree
$(a-b)\vec{s}$ for the two constants $a,b$. Since $\vec{s}$ has
infinite order, $\ker(\delta|_M)=\mathbb{Z}\mathbf{1}$ for
$\mathbf{1}=(1,\ldots,1)$. As $\deg_G(\ker\delta)=\mathbb{Z}\vec{s}$,
the construction of $B$ in Proposition~\ref{3corr}(1) gives
$\delta(M)=B$. Thus
\[
0\longrightarrow\mathbb{Z}\mathbf{1}\longrightarrow M
\xrightarrow{\delta} B\longrightarrow0.
\]

\begin{Lem}\label{dimerhomotopy}
Parallel paths $p,p'$ in $Q(C)$ are homotopic relative to their
endpoints in $T^d$ if and only if $\rho(p)/\rho(p')=D^n$ for some
$n\in\mathbb Z$. In particular, their labels agree precisely when
they are homotopic and have labels of equal total degree.
\end{Lem}
\begin{proof}
The exponent difference $m$ of the path labels belongs to $M$ by
Theorem~\ref{dimerreal}(\ref{dimerlabels}). The lifts from a common
source have endpoint difference $\delta(m)$, so homotopy is
equivalent to $m\in\mathbb{Z}\mathbf{1}$. Since $D$ has total degree
$d+2$, equality of total degrees then forces $m=0$.
\end{proof}

The $(d+1)$-torus $M_{\mathbb{R}}/M$ used for homotopy
presentations in \cite[\S~6]{Boc12} projects onto $T^d$ with kernel
the circle in the anticanonical direction $\mathbf{1}$. Craw and
Quintero V\'elez reconstruct a given algebraically consistent
dimer model from its quiver of sections by the corresponding
projection onto the two-torus \cite[\S~5.3]{CQV12}. For a dimer
model with a consistent $R$-charge, Bocklandt likewise
characterizes path equivalence by homotopy and charge
\cite[Theorem~6.6 and Lemma~7.4]{BocCons12}.

For $d=2$, we have $l=l'=1$, and a cut leaves one direction on each
edge of the square grid $\widehat Q$. By Definition~\ref{defcut}, a
clockwise traversal of each square has equally many cut horizontal
and cut vertical arrows. The retained sides form a directed boundary
or two directed paths from a source to the opposite sink. In the
latter case the mixed arrow from the sink to the source splits the
square into two directed triangles, and every mixed arrow is such a
diagonal. In the quotient tiling of $T^2$, each arrow separates faces
of opposite orientation, so the dual graph is a dimer model with
dual quiver $Q(C)$.

\section{Examples}

We use Theorem~\ref{upNCCRcorr} to classify the basic reflexive
modules giving toric NCCRs in several examples. Translating an upper
set by $q(\vec{a})$ translates all summand classes by $\vec{a}$ and
gives an isomorphic endomorphism algebra by Proposition~\ref{refhom}.
We give representatives up to simultaneous translation of the
summand classes and describe their endomorphism quivers.

First, we see the easiest example: the $3$-dimensional simple singularity of type $A_1$.

\begin{Ex}
Put $G:=\mathbb{Z}$ and $\vec{x}=\vec{y}=1,\vec{z}=\vec{w}=-1\in G$. We view $S:=k[x,y,z,w]$ as a $G$-graded $k$-algebra and define $R:=S_0=k[xz,xw,yz,yw]$. If we equip $H=G$ with our partial order, the quiver of $G$ becomes the following.
\[\xymatrix{
\cdots \ar@2@/^10pt/[r]^{x,y} \ar@2@/^-10pt/[r]_{-z,-w} & \circ \ar@2@/^10pt/[r]^{x,y} \ar@2@/^-10pt/[r]_{-z,-w} & \circ \ar@2@/^10pt/[r]^{x,y} \ar@2@/^-10pt/[r]_{-z,-w} & \circ \ar@2@/^10pt/[r]^{x,y} \ar@2@/^-10pt/[r]_{-z,-w} & \circ \ar@2@/^10pt/[r]^{x,y} \ar@2@/^-10pt/[r]_{-z,-w} & \cdots
}\]
Then there is just one kind of non-trivial upper sets in $G$ up to translations.
\[\xymatrix{
\circ \ar@2@/^10pt/[r]^{x,y} \ar@2@/^-10pt/[r]_{-z,-w} & \circ \ar@2@/^10pt/[r]^{x,y} \ar@2@/^-10pt/[r]_{-z,-w} & \circ \ar@2@/^10pt/[r]^{x,y} \ar@2@/^-10pt/[r]_{-z,-w} & \cdots
}\]
Remark that $\vec{s}=2$ holds. Here, the natural homomorphism $L\oplus L'=\mathbb{Z}^2\to G/\mathbb{Z}\vec{s}=\mathbb{Z}/2\mathbb{Z}$ sends $(a,b)$ to $a+b$. Thus $B=\{(a,b)\in L\oplus L'\mid a+b\in2\mathbb{Z}\}$ holds. If we let $\widehat{C}$ be the corresponding cut of $\widehat{Q}$ to the upper set in $H$, then $\widehat{Q}(\widehat{C})$ is as follows. Here, the black vertices correspond to elements of $B$.
\[\xymatrix{
\circ \ar[d]^w & \bullet \ar[r]^x \ar[l]^y & \circ \ar[d]^w & \bullet \ar[l]^y \\
\bullet \ar[r]^x & \circ \ar[u]^z \ar[d]^w & \bullet \ar[r]^x \ar[l]^y & \circ \ar[u]^z \ar[d]^w \\
\circ \ar[u]^z \ar[d]^w & \bullet \ar[r]^x \ar[l]^y & \circ \ar[u]^z \ar[d]^w & \bullet \ar[l]^y \\
\bullet \ar[r]^x & \circ \ar[u]^z & \bullet \ar[r]^x \ar[l]^y & \circ \ar[u]^z
}\]

Up to translation of the summand classes, there is one basic module
giving a toric NCCR. Its endomorphism algebra has the following
quiver $Q(C)$.
\[\xymatrix{
\circ \ar@2@/^10pt/[r]^{x,y} & \circ \ar@2@/^10pt/[l]^{z,w}
}\]
\end{Ex}

The next example has $\Cl(R)\cong\mathbb{Z}$ and a basic module
giving a toric NCCR whose summand classes form a nonconsecutive
set. Van den Bergh's construction uses an interval of classes
\cite[Theorem~8.9]{VdB04a}.

\begin{Ex}\label{exnoninterval}
Put $G:=\mathbb{Z}$ and $\vec{x}=2,\vec{y}=3,\vec{z}=-2,\vec{w}=-3\in G$. We view $S:=k[x,y,z,w]$ as a $G$-graded $k$-algebra and define $R:=S_0=k[xz,x^3w^2,y^2z^3,yw]$. Observe that $R$ is a $cA_4$ singularity. If we equip $H=G$ with our partial order, the quiver of $G$ becomes the following.
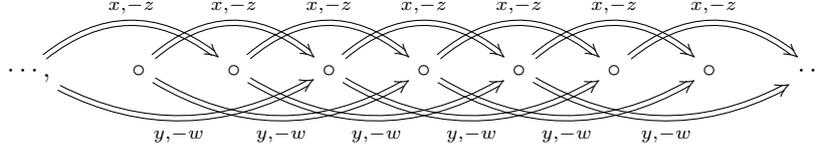
\[\xymatrix{
\cdots \ar@2@/^18pt/[rr]^{x,-z}, \ar@2@/^-18pt/[rrr]_{y,-w} & \circ \ar@2@/^18pt/[rr]^{x,-z} \ar@2@/^-18pt/[rrr]_{y,-w} & \circ \ar@2@/^18pt/[rr]^{x,-z} \ar@2@/^-18pt/[rrr]_{y,-w} & \circ \ar@2@/^18pt/[rr]^{x,-z} \ar@2@/^-18pt/[rrr]_{y,-w} & \circ \ar@2@/^18pt/[rr]^{x,-z} \ar@2@/^-18pt/[rrr]_{y,-w} & \circ \ar@2@/^18pt/[rr]^{x,-z} \ar@2@/^-18pt/[rrr]_{y,-w} & \circ \ar@2@/^18pt/[rr]^{x,-z} & \circ & \cdots
}\]
Then there are the following two kinds of non-trivial upper sets in $H$ up to translations.
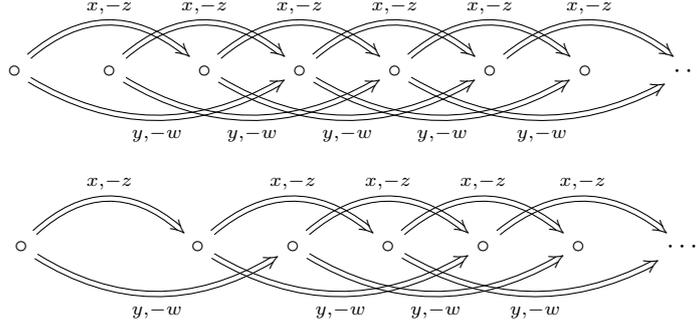
\[\xymatrix{
\circ \ar@2@/^18pt/[rr]^{x,-z} \ar@2@/^-18pt/[rrr]_{y,-w} & \circ \ar@2@/^18pt/[rr]^{x,-z} \ar@2@/^-18pt/[rrr]_{y,-w} & \circ \ar@2@/^18pt/[rr]^{x,-z} \ar@2@/^-18pt/[rrr]_{y,-w} & \circ \ar@2@/^18pt/[rr]^{x,-z} \ar@2@/^-18pt/[rrr]_{y,-w} & \circ \ar@2@/^18pt/[rr]^{x,-z} \ar@2@/^-18pt/[rrr]_{y,-w} & \circ \ar@2@/^18pt/[rr]^{x,-z} & \circ & \cdots
}\]
\[\xymatrix{
\circ \ar@2@/^18pt/[rr]^{x,-z} \ar@2@/^-18pt/[rrr]_{y,-w} & & \circ \ar@2@/^18pt/[rr]^{x,-z} \ar@2@/^-18pt/[rrr]_{y,-w} & \circ \ar@2@/^18pt/[rr]^{x,-z} \ar@2@/^-18pt/[rrr]_{y,-w} & \circ \ar@2@/^18pt/[rr]^{x,-z} \ar@2@/^-18pt/[rrr]_{y,-w} & \circ \ar@2@/^18pt/[rr]^{x,-z} & \circ & \cdots
}\]
Remark that $\vec{s}=5$ holds. Here, the natural homomorphism $L\oplus L'=\mathbb{Z}^2\to G/\mathbb{Z}\vec{s}=\mathbb{Z}/5\mathbb{Z}$ sends $(a,b)$ to $2a-2b$. Thus $B=\{(a,b)\in L\oplus L'\mid a-b\in5\mathbb{Z}\}$ holds. If we let $\widehat{C}_1,\widehat{C}_2$ be the corresponding cut of $\widehat{Q}$ to the two kinds of upper sets in $H$, then $\widehat{Q}(\widehat{C}_1)$ and $\widehat{Q}(\widehat{C}_2)$ are as follows. Here, the black vertices correspond to elements of $B$.
\[\xymatrix{
\circ \ar[d]^w & \bullet \ar[r]^x \ar[l]^y & \circ \ar[r]^x \ar[dl]^{yw} & \circ \ar[d]^w & \circ \ar[r]^x \ar[l]^y & \circ \ar[d]^w & \bullet \ar[l]^y \\
\bullet \ar[r]^x & \circ \ar[r]^x \ar[u]^z \ar[dl]^{yw} & \circ \ar[u]^z \ar[d]^w & \circ \ar[r]^x \ar[l]^y & \circ \ar[u]^z \ar[d]^w & \bullet \ar[r]^x \ar[l]^y & \circ \ar[u]^z \ar[dl]^{yw} \\
\circ \ar[r]^x \ar[u]^z & \circ \ar[u]^z \ar[d]^w & \circ \ar[r]^x \ar[l]^y & \circ \ar[u]^z \ar[d]^w & \bullet \ar[r]^x \ar[l]^y & \circ \ar[r]^x \ar[u]^z \ar[dl]^{yw} & \circ \ar[u]^z \ar[d]^w \\
\circ \ar[u]^z \ar[d]^w & \circ \ar[r]^x \ar[l]^y & \circ \ar[u]^z \ar[d]^w & \bullet \ar[r]^x \ar[l]^y & \circ \ar[r]^x \ar[u]^z \ar[dl]^{yw} & \circ \ar[u]^z \ar[d]^w & \circ \ar[l]^y \\
\circ \ar[r]^x & \circ \ar[u]^z \ar[d]^w & \bullet \ar[r]^x \ar[l]^y & \circ \ar[r]^x \ar[u]^z \ar[dl]^{yw} & \circ \ar[u]^z \ar[d]^w & \circ \ar[r]^x \ar[l]^y & \circ \ar[u]^z \ar[d]^w \\
\circ \ar[u]^z \ar[d]^w & \bullet \ar[r]^x \ar[l]^y & \circ \ar[r]^x \ar[u]^z \ar[dl]^{yw} & \circ \ar[u]^z \ar[d]^w & \circ \ar[r]^x \ar[l]^y & \circ \ar[u]^z \ar[d]^w & \bullet \ar[l]^y \\
\bullet \ar[r]^x & \circ \ar[r]^x \ar[u]^z & \circ \ar[u]^z & \circ \ar[r]^x \ar[l]^y & \circ \ar[u]^z & \bullet \ar[r]^x \ar[l]^y & \circ \ar[u]^z
}\]

\[\xymatrix{
\circ \ar[d]^w & \bullet \ar[r]^x \ar[l]^y & \circ \ar[r]^x \ar[dl]^{yw} & \circ \ar[r]^x \ar[dl]^{yw} & \circ \ar[d]^w & \circ \ar[l]^y \ar[d]^w &\bullet \ar[l]^y \\
\bullet \ar[r]^x & \circ \ar[r]^x \ar[u]^z \ar[dl]^{yw} & \circ \ar[r]^x \ar[u]^z \ar[dl]^{yw} & \circ \ar[u]^z \ar[d]^w & \circ \ar[l]^y \ar[d]^w \ar[ur]^{xz} & \bullet \ar[r]^x \ar[l]^y & \circ \ar[u]^z \ar[dl]^{yw} \\
\circ \ar[r]^x \ar[u]^z & \circ \ar[r]^x \ar[u]^z \ar[dl]^{yw} & \circ \ar[u]^z \ar[d]^w & \circ \ar[l]^y \ar[d]^w \ar[ur]^{xz} & \bullet \ar[r]^x \ar[l]^y & \circ \ar[r]^x \ar[u]^z \ar[dl]^{yw} & \circ \ar[u]^z \ar[dl]^{yw} \\
\circ \ar[r]^x \ar[u]^z & \circ \ar[u]^z \ar[d]^w & \circ \ar[l]^y \ar[d]^w \ar[ur]^{xz} & \bullet \ar[r]^x \ar[l]^y & \circ \ar[r]^x \ar[u]^z \ar[dl]^{yw} & \circ \ar[r]^x \ar[u]^z \ar[dl]^{yw} & \circ \ar[u]^z \ar[d]^w \\
\circ \ar[u]^z \ar[d]^w & \circ \ar[l]^y \ar[d]^w \ar[ur]^{xz} & \bullet \ar[r]^x \ar[l]^y & \circ \ar[r]^x \ar[u]^z \ar[dl]^{yw} & \circ \ar[r]^x \ar[u]^z \ar[dl]^{yw} & \circ \ar[u]^z \ar[d]^w & \circ \ar[l]^y \ar[d]^w \\
\circ \ar[d]^w \ar[ur]^{xz} & \bullet \ar[r]^x \ar[l]^y & \circ \ar[r]^x \ar[u]^z \ar[dl]^{yw} & \circ \ar[r]^x \ar[u]^z \ar[dl]^{yw} & \circ \ar[u]^z \ar[d]^w & \circ \ar[l]^y \ar[d]^w \ar[ur]^{xz} & \bullet \ar[l]^y \\
\bullet \ar[r]^x & \circ \ar[r]^x \ar[u]^z & \circ \ar[r]^x \ar[u]^z & \circ \ar[u]^z & \circ \ar[l]^y \ar[ur]^{xz} & \bullet \ar[r]^x \ar[l]^y & \circ \ar[u]^z
}\]
Up to translation of the summand classes, there are two basic
modules giving toric NCCRs. Their endomorphism algebras have the
following quivers $Q(C_1)$ and $Q(C_2)$. Mutations of the
corresponding upper sets connect these modules by
Theorem~\ref{toricmut}.
\[\xymatrix{
\circ \ar@/^15pt/[rr]_x \ar@/^24pt/[rrr]^y & \circ \ar@/^15pt/[rr]_x \ar@/^24pt/[rrr]^y & \circ \ar@/^15pt/[rr]_x \ar@/^15pt/[ll]_z \ar@(ur,dr)^{yw} & \circ \ar@/^15pt/[ll]_z \ar@/^24pt/[lll]^w & \circ \ar@/^15pt/[ll]_z \ar@/^24pt/[lll]^w
}\]
\[\xymatrix{
\circ \ar@/^15pt/[rr]_x \ar@/^21pt/[rrr]^y & & \circ \ar@/^15pt/[rr]_x \ar@/^15pt/[ll]_z \ar@(dl,ul)^{yw} & \circ \ar@/^21pt/[rrr]^y \ar@/^21pt/[lll]^w \ar@(ul,ur)^{xz} & \circ \ar@/^15pt/[rr]_x \ar@/^15pt/[ll]_z \ar@(ur,dr)^{yw} & & \circ \ar@/^15pt/[ll]_z \ar@/^21pt/[lll]^w
}\]
Observe that the second one is not obtained in \cite{VdB04a}.
\end{Ex}

We see examples such that $\Cl(R)$ has a torsion part.

\begin{Ex}\label{extors}
Put $G:=\mathbb{Z}\oplus(\mathbb{Z}/2\mathbb{Z}), \vec{x}=(1,0), \vec{y}=(1,1), \vec{z}=(-1,0), \vec{w}=(-1,1)\in G$. We view $S:=k[x,y,z,w]$ as a $G$-graded $k$-algebra and define $R:=S_0=k[xz,x^2w^2,y^2z^2,yw]$. Observe that $R$ is a $cA_3$ singularity. If we equip $H=G$ with our partial order, the quiver of $G$ becomes the following.
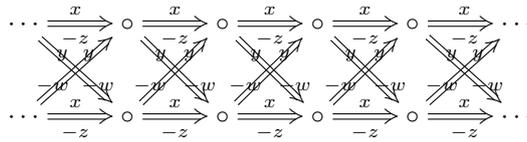
\[\xymatrix{
\cdots \ar@2[r]^x_{-z} \ar@2[dr]^y_{-w} & \circ \ar@2[r]^x_{-z} \ar@2[dr]^y_{-w}  & \circ \ar@2[r]^x_{-z} \ar@2[dr]^y_{-w}  & \circ \ar@2[r]^x_{-z} \ar@2[dr]^y_{-w}  & \circ \ar@2[r]^x_{-z} \ar@2[dr]^y_{-w}  & \cdots \\
\cdots \ar@2[r]^x_{-z} \ar@2[ur]^y_{-w} & \circ \ar@2[r]^x_{-z} \ar@2[ur]^y_{-w} & \circ \ar@2[r]^x_{-z} \ar@2[ur]^y_{-w} & \circ \ar@2[r]^x_{-z} \ar@2[ur]^y_{-w} & \circ \ar@2[r]^x_{-z} \ar@2[ur]^y_{-w}  & \cdots
}\]
Then there are the following two kinds of non-trivial upper sets in $G$ up to translations.
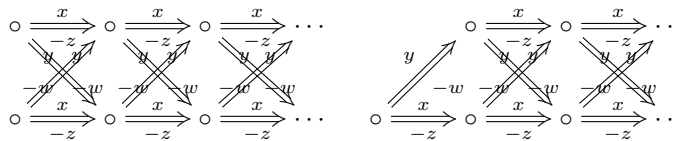
\[\begin{array}{c c}
\xymatrix{
\circ \ar@2[r]^x_{-z} \ar@2[dr]^y_{-w} & \circ \ar@2[r]^x_{-z} \ar@2[dr]^y_{-w} & \circ \ar@2[r]^x_{-z} \ar@2[dr]^y_{-w} & \cdots \\
\circ \ar@2[r]^x_{-z} \ar@2[ur]^y_{-w} & \circ \ar@2[r]^x_{-z} \ar@2[ur]^y_{-w} & \circ \ar@2[r]^x_{-z} \ar@2[ur]^y_{-w} & \cdots
}&\xymatrix{
 & \circ \ar@2[r]^x_{-z} \ar@2[dr]^y_{-w}  & \circ \ar@2[r]^x_{-z} \ar@2[dr]^y_{-w}  & \cdots \\
\circ \ar@2[r]^x_{-z} \ar@2[ur]^y_{-w} & \circ \ar@2[r]^x_{-z} \ar@2[ur]^y_{-w} & \circ \ar@2[r]^x_{-z} \ar@2[ur]^y_{-w}  & \cdots
}\end{array}\]
Remark that $\vec{s}=(2,1)$ holds. Here, the natural homomorphism $L\oplus L'=\mathbb{Z}^2\to G/\mathbb{Z}\vec{s}=\mathbb{Z}/4\mathbb{Z}$ sends $(a,b)$ to $a-b$. Thus $B=\{(a,b)\in L\oplus L'\mid a-b\in4\mathbb{Z}\}$ holds. If we let $\widehat{C}_1,\widehat{C}_2$ be the corresponding cut of $\widehat{Q}$ to the two kinds of upper sets in $H$, then $\widehat{Q}(\widehat{C}_1)$ and $\widehat{Q}(\widehat{C}_2)$ are as follows. Here, the black vertices correspond to elements of $B$.
\[\xymatrix{
\circ \ar[d]^w & \bullet \ar[r]^x \ar[l]^y & \circ \ar[d]^w & \circ \ar[r]^x \ar[l]^y & \circ \ar[d]^w &\bullet \ar[l]^y \\
\bullet \ar[r]^x & \circ \ar[u]^z \ar[d]^w & \circ \ar[r]^x \ar[l]^y & \circ \ar[u]^z \ar[d]^w & \bullet \ar[r]^x \ar[l]^y & \circ \ar[u]^z \ar[d]^w \\
\circ \ar[u]^z \ar[d]^w & \circ \ar[r]^x \ar[l]^y & \circ \ar[u]^z \ar[d]^w & \bullet \ar[r]^x \ar[l]^y & \circ \ar[u]^z \ar[d]^w & \circ \ar[l]^y \\
\circ \ar[r]^x & \circ \ar[u]^z \ar[d]^w & \bullet \ar[r]^x \ar[l]^y & \circ \ar[u]^z \ar[d]^w & \circ \ar[r]^x \ar[l]^y & \circ \ar[u]^z \ar[d]^w \\
\circ \ar[u]^z \ar[d]^w & \bullet \ar[r]^x \ar[l]^y & \circ \ar[u]^z \ar[d]^w & \circ \ar[r]^x \ar[l]^y & \circ \ar[u]^z \ar[d]^w & \bullet \ar[l]^y \\
\bullet \ar[r]^x & \circ \ar[u]^z & \circ \ar[r]^x \ar[l]^y & \circ \ar[u]^z & \bullet \ar[r]^x \ar[l]^y & \circ \ar[u]^z
}\xymatrix{
\circ \ar[d]^w & \bullet \ar[r]^x \ar[l]^y & \circ \ar[r]^x \ar[dl]^{yw} & \circ \ar[d]^w & \circ \ar[l]^y \ar[d]^w & \bullet \ar[l]^y \\
\bullet \ar[r]^x & \circ \ar[r]^x \ar[u]^z \ar[dl]^{yw} & \circ \ar[u]^z \ar[d]^w & \circ \ar[l]^y \ar[d]^w \ar[ur]^{xz} & \bullet \ar[r]^x \ar[l]^y & \circ \ar[u]^z \ar[dl]^{yw} \\
\circ \ar[r]^x \ar[u]^z & \circ \ar[u]^z \ar[d]^w & \circ \ar[l]^y \ar[d]^w \ar[ur]^{xz} & \bullet \ar[r]^x \ar[l]^y & \circ \ar[r]^x \ar[u]^z \ar[dl]^{yw} & \circ \ar[u]^z \ar[d]^w \\
\circ \ar[u]^z \ar[d]^w & \circ \ar[l]^y \ar[d]^w \ar[ur]^{xz} & \bullet \ar[r]^x \ar[l]^y & \circ \ar[r]^x \ar[u]^z \ar[dl]^{yw} & \circ \ar[u]^z \ar[d]^w & \circ \ar[l]^y \ar[d]^w \\
\circ \ar[d]^w \ar[ur]^{xz} & \bullet \ar[r]^x \ar[l]^y & \circ \ar[r]^x \ar[u]^z \ar[dl]^{yw} & \circ \ar[u]^z \ar[d]^w & \circ \ar[l]^y \ar[d]^w \ar[ur]^{xz} & \bullet \ar[l]^y \\
\bullet \ar[r]^x & \circ \ar[r]^x \ar[u]^z & \circ \ar[u]^z & \circ \ar[l]^y \ar[ur]^{xz} & \bullet \ar[r]^x \ar[l]^y & \circ \ar[u]^z
}\]
Up to translation of the summand classes, there are two basic
modules giving toric NCCRs. Their endomorphism algebras have the
following quivers $Q(C_1)$ and $Q(C_2)$. Mutations of the
corresponding upper sets connect these modules by
Theorem~\ref{toricmut}.
\[\begin{array}{c c}
\xymatrix{
\circ \ar@<0.25ex>[r]^x \ar@<0.25ex>[dr]^y & \circ \ar@<0.25ex>[l]^z \ar@<0.25ex>[dl]^w \\
\circ \ar@<0.25ex>[r]^x \ar@<0.25ex>[ur]^y & \circ \ar@<0.25ex>[l]^z \ar@<0.25ex>[ul]^w
}&\xymatrix{
 & \circ \ar@<0.25ex>[dr]^y \ar@<0.25ex>[dl]^w \ar@(ul,ur)^{xz} \\
\circ \ar@<0.25ex>[r]^x \ar@<0.25ex>[ur]^y & \circ \ar@<0.25ex>[r]^x \ar@<0.25ex>[l]^z \ar@(dr,dl)^{yw} & \circ \ar@<0.25ex>[l]^z \ar@<0.25ex>[ul]^w
}\end{array}\]
\end{Ex}

Next, we see a $4$-dimensional example.

\begin{Ex}
Put $G:=\mathbb{Z}, \vec{x}=\vec{y}=1, \vec{z}=2, \vec{u}=-1,\vec{v}=-3\in G$. We view $S:=k[x,y,z,u,v]$ as a $G$-graded $k$-algebra and define $R:=S_0=k[xu,yu,zu^2,x^3v,x^2yv,xy^2v,y^3v,z^3v^2,xzv,yzv,z^2uv]$. If we equip $H=G$ with our partial order, the quiver of $G$ becomes the following.
\[\xymatrix{
\cdots \ar@3[r]^{x,y}_{-u} \ar@/^18pt/[rr]^z \ar@/^-18pt/[rrr]_{-v} & \circ \ar@3[r]^{x,y}_{-u} \ar@/^18pt/[rr]^z \ar@/^-18pt/[rrr]_{-v} & \circ \ar@3[r]^{x,y}_{-u} \ar@/^18pt/[rr]^z \ar@/^-18pt/[rrr]_{-v} & \circ \ar@3[r]^{x,y}_{-u} \ar@/^18pt/[rr]^z \ar@/^-18pt/[rrr]_{-v} & \circ \ar@3[r]^{x,y}_{-u} \ar@/^18pt/[rr]^z & \circ \ar@3[r]^{x,y}_{-u} & \cdots
}\]
Up to translation, $I=\mathbb{Z}_{\geq0}$ is the unique non-trivial upper set.
\[\xymatrix{
\circ \ar@3[r]^{x,y}_{-u} \ar@/^18pt/[rr]^z \ar@/^-18pt/[rrr]_{-v} & \circ \ar@3[r]^{x,y}_{-u} \ar@/^18pt/[rr]^z \ar@/^-18pt/[rrr]_{-v} & \circ \ar@3[r]^{x,y}_{-u} \ar@/^18pt/[rr]^z \ar@/^-18pt/[rrr]_{-v} & \circ \ar@3[r]^{x,y}_{-u} \ar@/^18pt/[rr]^z & \circ \ar@3[r]^{x,y}_{-u} & \cdots
}\]
Here $\vec{s}=4$. In the figure, the arrow vectors for $x,y,u$ form the ordered basis of $L\oplus L'\cong\mathbb{Z}^3$. The map to $G/\mathbb{Z}\vec{s}\cong\mathbb{Z}/4\mathbb{Z}$ sends $(a,b,c)$ to the class of $a+b-c$, so
\[
B=\{(a,b,c)\in\mathbb{Z}^3\mid a+b-c\in4\mathbb{Z}\}.
\]
The figure shows $\widehat{Q}(\widehat{C})$ for the cut $\widehat{C}$ corresponding to $I$. Black vertices mark the elements of $B$.

\begin{center}
\begin{tikzpicture}[
  x={(2.5cm,0cm)},
  y={(-1.5cm,1cm)},
  z={(0cm,2.5cm)},
  >={Stealth[length=2.2mm,width=1.8mm]},
  line cap=round,
  line join=round,
  edge/.style={->, line width=0.95pt},
  rededge/.style={->, red!85!black, line width=1.35pt},
  vertex/.style={circle, fill=black, inner sep=1.25pt},
  every node/.style={font=\scriptsize}
]

\def\N{4}

\foreach \i in {0,...,\N}{
  \foreach \j in {0,...,2}{
    \foreach \k in {0,...,\N}{
      \coordinate (\i-\j-\k) at (\i,\j,\k);
    }
  }
}

\foreach \i in {0,...,\N}{
  \foreach \k in {0,...,\N}{
    \pgfmathtruncatemacro{\labf}{mod(\i+2-\k+40,4)}
    
    \ifnum\labf=0
      \ifnum\i<\N
        \pgfmathtruncatemacro{\ip}{\i+1}
        \draw[->, thick] ($(\i-2-\k)+(0.04,0,0)$) -- ($(\ip-2-\k)+(-0.04,0,0)$);
      \fi
      \ifnum\i>0
        \draw[->, thick] ($(\i-2-\k)+(-0.04,-0.04,0)$) -- ($(\i-2-\k)+(-0.96,-0.96,0)$);
      \fi
    \fi
    
    \ifnum\labf=1
      \ifnum\i<\N
        \pgfmathtruncatemacro{\ip}{\i+1}
        \draw[->, thick] ($(\i-2-\k)+(0.04,0,0)$) -- ($(\ip-2-\k)+(-0.04,0,0)$);
      \fi
      \ifnum\k<\N
        \pgfmathtruncatemacro{\kp}{\k+1}
        \draw[->, thick] ($(\i-2-\k)+(0,0,0.04)$) -- ($(\i-2-\kp)+(0,0,-0.04)$);
      \fi
      \ifnum\i>0
        \draw[->, thick] ($(\i-2-\k)+(-0.04,-0.04,0)$) -- ($(\i-2-\k)+(-0.96,-0.96,0)$);
      \fi
    \fi
    
    \ifnum\labf=2
      \ifnum\i<\N
        \pgfmathtruncatemacro{\ip}{\i+1}
        \draw[->, thick] ($(\i-2-\k)+(0.04,0,0)$) -- ($(\ip-2-\k)+(-0.04,0,0)$);
      \fi
      \ifnum\k<\N
        \pgfmathtruncatemacro{\kp}{\k+1}
        \draw[->, thick] ($(\i-2-\k)+(0,0,0.04)$) -- ($(\i-2-\kp)+(0,0,-0.04)$);
      \fi
      \ifnum\i>0
      \ifnum\k>0
        \draw[white, line width=6pt] ($(\i-2-\k)+(-0.15,-0.15,-0.15)$) -- ($(\i-2-\k)+(-0.85,-0.85,-0.85)$);
        \draw[->, thick] ($(\i-2-\k)+(-0.04,-0.04,-0.04)$) -- ($(\i-2-\k)+(-0.96,-0.96,-0.96)$);
      \fi
      \fi
    \fi
    
    \ifnum\labf=3
      \ifnum\k<\N
        \pgfmathtruncatemacro{\kp}{\k+1}
        \draw[->, thick] ($(\i-2-\k)+(0,0,0.04)$) -- ($(\i-2-\kp)+(0,0,-0.04)$);
      \fi
      \ifnum\k>0
        \pgfmathtruncatemacro{\kn}{\k-1}
        \draw[->, thick] ($(\i-2-\k)+(0,0,-0.04)$) -- ($(\i-2-\kn)+(0,0,0.04)$);
      \fi
    \fi
  }
}

\foreach \i in {0,...,\N}{
  \foreach \j in {1}{
    \foreach \k in {0,...,\N}{
      \pgfmathtruncatemacro{\labf}{mod(\i+\j-\k+40,4)}
      
      \ifnum\labf=0
        \ifnum\i<\N
          \pgfmathtruncatemacro{\ip}{\i+1}
          \draw[white, line width=6pt] ($(\i-\j-\k)+(0.05,0,0)$) -- ($(\ip-\j-\k)+(-0.05,0,0)$);
          \draw[->, thick] ($(\i-\j-\k)+(0.04,0,0)$) -- ($(\ip-\j-\k)+(-0.04,0,0)$);
        \fi
        \pgfmathtruncatemacro{\jp}{\j+1}
        \draw[white, line width=6pt] ($(\i-\j-\k)+(0,0.15,0)$) -- ($(\i-\jp-\k)+(0,-0.15,0)$);
        \draw[->, thick] ($(\i-\j-\k)+(0,0.04,0)$) -- ($(\i-\jp-\k)+(0,-0.04,0)$);
        \ifnum\i>0
        \ifnum\j>0
          \draw[->, thick] ($(\i-\j-\k)+(-0.04,-0.04,0)$) -- ($(\i-\j-\k)+(-0.96,-0.96,0)$);
        \fi
        \fi
      \fi
       
      \ifnum\labf=1
        \ifnum\i<\N
          \pgfmathtruncatemacro{\ip}{\i+1}
          \draw[white, line width=6pt] ($(\i-\j-\k)+(0.05,0,0)$) -- ($(\ip-\j-\k)+(-0.05,0,0)$);
          \draw[->, thick] ($(\i-\j-\k)+(0.04,0,0)$) -- ($(\ip-\j-\k)+(-0.04,0,0)$);
        \fi
        \ifnum\k<\N
          \pgfmathtruncatemacro{\kp}{\k+1}
          \draw[white, line width=6pt] ($(\i-\j-\k)+(0,0,0.2)$) -- ($(\i-\j-\kp)+(0,0,-0.05)$);
          \draw[->, thick] ($(\i-\j-\k)+(0,0,0.04)$) -- ($(\i-\j-\kp)+(0,0,-0.04)$);
        \fi
        \pgfmathtruncatemacro{\jp}{\j+1}
        \draw[white, line width=6pt] ($(\i-\j-\k)+(0,0.15,0)$) -- ($(\i-\jp-\k)+(0,-0.15,0)$);
        \draw[->, thick] ($(\i-\j-\k)+(0,0.04,0)$) -- ($(\i-\jp-\k)+(0,-0.04,0)$);
        \ifnum\i>0
        \ifnum\j>0
          \draw[->, thick] ($(\i-\j-\k)+(-0.04,-0.04,0)$) -- ($(\i-\j-\k)+(-0.96,-0.96,0)$);
        \fi
        \fi
     \fi
      
      \ifnum\labf=2
        \ifnum\i<\N
          \pgfmathtruncatemacro{\ip}{\i+1}
          \draw[white, line width=6pt] ($(\i-\j-\k)+(0.1,0,0)$) -- ($(\ip-\j-\k)+(-0.05,0,0)$);
          \draw[->, thick] ($(\i-\j-\k)+(0.04,0,0)$) -- ($(\ip-\j-\k)+(-0.04,0,0)$);
        \fi
        \ifnum\k<\N
          \pgfmathtruncatemacro{\kp}{\k+1}
          \draw[white, line width=6pt] ($(\i-\j-\k)+(0,0,0.1)$) -- ($(\i-\j-\kp)+(0,0,-0.05)$);
          \draw[->, thick] ($(\i-\j-\k)+(0,0,0.04)$) -- ($(\i-\j-\kp)+(0,0,-0.04)$);
        \fi
        \pgfmathtruncatemacro{\jp}{\j+1}
        \draw[white, line width=6pt] ($(\i-\j-\k)+(0,0.15,0)$) -- ($(\i-\jp-\k)+(0,-0.15,0)$);
        \draw[->, thick] ($(\i-\j-\k)+(0,0.04,0)$) -- ($(\i-\jp-\k)+(0,-0.04,0)$);
        \ifnum\i>0
        \ifnum\k>0
          \draw[white, line width=6pt] ($(\i-\j-\k)+(-0.15,-0.15,-0.15)$) -- ($(\i-\j-\k)+(-0.85,-0.85,-0.85)$);
          \draw[->, thick] ($(\i-\j-\k)+(-0.04,-0.04,-0.04)$) -- ($(\i-\j-\k)+(-0.96,-0.96,-0.96)$);
        \fi
        \fi
      \fi
          
      \ifnum\labf=3
        \ifnum\k<\N
          \pgfmathtruncatemacro{\kp}{\k+1}
          \draw[white, line width=6pt] ($(\i-\j-\k)+(0,0,0.1)$) -- ($(\i-\j-\kp)+(0,0,-0.05)$);
          \draw[->, thick] ($(\i-\j-\k)+(0,0,0.04)$) -- ($(\i-\j-\kp)+(0,0,-0.04)$);
        \fi
        \ifnum\k>0
          \pgfmathtruncatemacro{\kn}{\k-1}
          \draw[white, line width=6pt] ($(\i-\j-\k)+(0,0,-0.05)$) -- ($(\i-\j-\kn)+(0,0,0.1)$);
          \draw[->, thick] ($(\i-\j-\k)+(0,0,-0.04)$) -- ($(\i-\j-\kn)+(0,0,0.04)$);
        \fi
      \fi
     }
  }
}

\foreach \i in {0,...,\N}{
  \foreach \j in {0}{
    \foreach \k in {0,...,\N}{
      \pgfmathtruncatemacro{\labf}{mod(\i+\j-\k+40,4)}
      
      \ifnum\labf=0
        \ifnum\i<\N
          \pgfmathtruncatemacro{\ip}{\i+1}
          \draw[white, line width=6pt] ($(\i-\j-\k)+(0.05,0,0)$) -- ($(\ip-\j-\k)+(-0.05,0,0)$);
          \draw[->, thick] ($(\i-\j-\k)+(0.04,0,0)$) -- ($(\ip-\j-\k)+(-0.04,0,0)$);
        \fi
        \pgfmathtruncatemacro{\jp}{\j+1}
        \draw[white, line width=6pt] ($(\i-\j-\k)+(0,0.15,0)$) -- ($(\i-\jp-\k)+(0,-0.15,0)$);
        \draw[->, thick] ($(\i-\j-\k)+(0,0.04,0)$) -- ($(\i-\jp-\k)+(0,-0.04,0)$);
       \fi
       
      \ifnum\labf=1
        \ifnum\i<\N
          \pgfmathtruncatemacro{\ip}{\i+1}
          \draw[white, line width=6pt] ($(\i-\j-\k)+(0.05,0,0)$) -- ($(\ip-\j-\k)+(-0.05,0,0)$);
          \draw[->, thick] ($(\i-\j-\k)+(0.04,0,0)$) -- ($(\ip-\j-\k)+(-0.04,0,0)$);
        \fi
        \ifnum\k<\N
          \pgfmathtruncatemacro{\kp}{\k+1}
          \draw[white, line width=6pt] ($(\i-\j-\k)+(0,0,0.2)$) -- ($(\i-\j-\kp)+(0,0,-0.05)$);
          \draw[->, thick] ($(\i-\j-\k)+(0,0,0.04)$) -- ($(\i-\j-\kp)+(0,0,-0.04)$);
        \fi
        \pgfmathtruncatemacro{\jp}{\j+1}
        \draw[white, line width=6pt] ($(\i-\j-\k)+(0,0.15,0)$) -- ($(\i-\jp-\k)+(0,-0.15,0)$);
        \draw[->, thick] ($(\i-\j-\k)+(0,0.04,0)$) -- ($(\i-\jp-\k)+(0,-0.04,0)$);
      \fi
      
      \ifnum\labf=2
        \ifnum\i<\N
          \pgfmathtruncatemacro{\ip}{\i+1}
          \draw[white, line width=6pt] ($(\i-\j-\k)+(0.1,0,0)$) -- ($(\ip-\j-\k)+(-0.05,0,0)$);
          \draw[->, thick] ($(\i-\j-\k)+(0.04,0,0)$) -- ($(\ip-\j-\k)+(-0.04,0,0)$);
        \fi
        \ifnum\k<\N
          \pgfmathtruncatemacro{\kp}{\k+1}
          \draw[white, line width=6pt] ($(\i-\j-\k)+(0,0,0.1)$) -- ($(\i-\j-\kp)+(0,0,-0.05)$);
          \draw[->, thick] ($(\i-\j-\k)+(0,0,0.04)$) -- ($(\i-\j-\kp)+(0,0,-0.04)$);
        \fi
        \pgfmathtruncatemacro{\jp}{\j+1}
        \draw[white, line width=6pt] ($(\i-\j-\k)+(0,0.15,0)$) -- ($(\i-\jp-\k)+(0,-0.15,0)$);
        \draw[->, thick] ($(\i-\j-\k)+(0,0.04,0)$) -- ($(\i-\jp-\k)+(0,-0.04,0)$);
      \fi
          
      \ifnum\labf=3
        \ifnum\k<\N
          \pgfmathtruncatemacro{\kp}{\k+1}
          \draw[white, line width=6pt] ($(\i-\j-\k)+(0,0,0.1)$) -- ($(\i-\j-\kp)+(0,0,-0.05)$);
          \draw[->, thick] ($(\i-\j-\k)+(0,0,0.04)$) -- ($(\i-\j-\kp)+(0,0,-0.04)$);
        \fi
        \ifnum\k>0
          \pgfmathtruncatemacro{\kn}{\k-1}
          \draw[white, line width=6pt] ($(\i-\j-\k)+(0,0,-0.05)$) -- ($(\i-\j-\kn)+(0,0,0.1)$);
          \draw[->, thick] ($(\i-\j-\k)+(0,0,-0.04)$) -- ($(\i-\j-\kn)+(0,0,0.04)$);
        \fi
      \fi
     }
  }
}

\foreach \i in {0,...,\N}{
  \foreach \j in {0,...,2}{
    \foreach \k in {0,...,\N}{
    \pgfmathtruncatemacro{\labf}{mod(\i+\j-\k+40,4)}
      \ifnum\labf=0
        \node at (\i-\j-\k) {$\bullet$};
      \else
        \node at (\i-\j-\k) {$\circ$};
      \fi
    }
  }
}

\end{tikzpicture}
\end{center}

Here, each arrow of each direction represents the following monomials.

\begin{center}
\begin{tikzpicture}[
  x={(2.5cm,0cm)},
  y={(-1.5cm,1cm)},
  z={(0cm,2.5cm)},
  >={Stealth[length=2.2mm,width=1.8mm]},
  line cap=round,
  line join=round,
  edge/.style={->, line width=0.95pt},
  rededge/.style={->, red!85!black, line width=1.35pt},
  vertex/.style={circle, fill=black, inner sep=1.25pt},
  every node/.style={font=\scriptsize}
]

\node at (0,0,0) {$\circ$};

\draw[->, thick] (0.04,0,0) -- node[midway,above] {$x$} (0.96,0,0);
\draw[->, thick] (0,0.04,0) -- node[midway,above] {$y$} (0,0.96,0);
\draw[->, thick] (-0.04,-0.04,0) -- node[midway,above] {$z$} (-0.96,-0.96,0);
\draw[->, thick] (0,0,0.04) -- node[midway,right] {$u$} (0,0,0.96);
\draw[->, thick] (0,0,-0.04) -- node[midway,right] {$v$} (0,0,-0.96);
\draw[->, thick] (-0.04,-0.04,-0.04) -- node[midway,left] {$zv$} (-0.96,-0.96,-0.96);
\end{tikzpicture}
\end{center}

Up to translation of the summand classes, there is one basic module
giving a toric NCCR. Its endomorphism algebra has the following
quiver $Q(C)$.
\[\xymatrix{
\circ \ar@2@/^5pt/[r]^{x,y} \ar@/^18pt/[rr]^z & \circ \ar@2@/^5pt/[r]^{x,y} \ar@/^18pt/[rr]^z \ar@/^5pt/[l]^u & \circ \ar@2@/^5pt/[r]^{x,y} \ar@2@/^5pt/[l]^{u,zv} & \circ \ar@/^5pt/[l]^u \ar@/^18pt/[lll]^v
}\]

\end{Ex}

\begin{Ex}\label{example:non-pyramid}
{\samepage
Let $G=\mathbb{Z}\oplus(\mathbb{Z}/4\mathbb{Z})$\label{expyramid} and grade $S=k[x,y,z,w,u]$ by
\[
\vec{x}=(1,0),\quad \vec{y}=(1,1),\quad
\vec{z}=(-1,0),\quad \vec{w}=(-1,1),\quad \vec{u}=(0,2).
\]
The weights generate $G$ and sum to zero, so $R:=S_0$ is Gorenstein with $\Cl(R)\cong G$. Here $l+l'=2<d=3$, so the polytope is a pyramid.\par}

The map $q\colon G\to H=G/\mathbb{Z}\vec{u}$ reduces the second coordinate modulo $2$ and has kernel of order two. Set $A=xz$, $B=(xw)^2$, $C=(yz)^2$ and $D=yw$. For the $H$-grading, the degree-zero part is $k[A,B,C,D,u]$ by Example \ref{extors}. In $G$, the last four generators have degree $\vec{u}$, and $A$ has degree zero. Pairing the generators of degree $\vec{u}$ and using $BC=A^2D^2$ gives
\[
R=k[A,B^2,C^2,D^2,u^2,BD,CD,Bu,Cu,Du].
\]

The order on $H$ is as in Example \ref{extors}, with $s=(2,\bar{1})$. Up to translation, the upper sets are represented by $I_0=\mathbb{Z}_{\geq0}\times(\mathbb{Z}/2\mathbb{Z})$ and $I_1=I_0\setminus\{(0,\bar{0})\}$. Put $T_i=q^{-1}(J(I_i))$. Then
\[
T_0=\{0,1\}\times(\mathbb{Z}/4\mathbb{Z}),\qquad
T_1=(\{0,2\}\times\{1,3\})\cup(\{1\}\times(\mathbb{Z}/4\mathbb{Z})).
\]
Since $I_1=\mu^-_{(0,\bar{0})}(I_0)$, Theorem \ref{toricmut} relates $M_{T_0}$ and $M_{T_1}$ by mutation at $M_{T_0\cap T_1}$.

In the quiver $Q(T_i)$ of $\End_R(M_{T_i})$, a label $f$ shifts a vertex $(n,r)$ by $\deg_G f$, with $r$ read modulo $4$. A variable gives an arrow whenever its shift stays in $T_i$. By Proposition \ref{refhom}, maps between the summands have a monomial basis. Factor a nonconstant monomial by taking an available variable arrow first. If none is available, the source is in the middle column of $T_1$, since the unavailable variables at an outer column all point outwards. The unavailable pair is $x,z$ at even $r$ and $y,w$ at odd $r$. A power of either variable alone cannot reach a vertex of $T_1$, so the monomial has the corresponding mixed factor. This gives the quivers below, with relations identifying parallel paths with the same monomial label.

The rows are ordered by $r=0,1,2,3$ from top to bottom, and the columns by $n$ from left to right starting at $0$. Each double-headed edge denotes one arrow in each direction for each label.
\[
\begin{array}{@{}c@{\hspace{14mm}}c@{}}
\xymatrix@C=46pt@R=27pt{
 \circ \ar@<.4ex>[r]^x \ar[dr]_(.33){y} \ar@{<->}@/_13pt/[dd]_u &
 \circ \ar@<.4ex>[l]^z \ar[dl]^(.33){w} \ar@{<->}@/^13pt/[dd]^u \\
 \circ \ar@<.4ex>[r]^x \ar[dr]_(.33){y} \ar@{<->}@/_21pt/[dd]^u &
 \circ \ar@<.4ex>[l]^z \ar[dl]^(.33){w} \ar@{<->}@/^21pt/[dd]_u \\
 \circ \ar@<.4ex>[r]^x \ar[dr]_(.33){y} &
 \circ \ar@<.4ex>[l]^z \ar[dl]^(.33){w} \\
 \circ \ar@<.4ex>[r]^x \ar@`{+/ul 142pt/,p+/ul 72pt/}[uuur]^(.34){y} &
 \circ \ar@<.4ex>[l]^z \ar@`{+/ur 142pt/,p+/ur 72pt/}[uuul]_(.34){w}
}
&
\xymatrix@C=46pt@R=27pt{
 & \circ \ar[dr]^(.57){y} \ar[dl]_(.57){w}
   \ar@(ul,ur)^{xz} \ar@{<->}@/^10pt/[dd]^(.37){u} & \\
 \circ \ar@<.4ex>[r]^x \ar[dr]_(.33){y} \ar@{<->}[dd]_u &
 \circ \ar@<.4ex>[r]^x \ar@<.4ex>[l]^z \ar@2{<->}@/_17pt/[dd]_{u,yw} &
 \circ \ar@<.4ex>[l]^z \ar[dl]^(.33){w} \ar@{<->}[dd]^u \\
 & \circ \ar[dr]^(.57){y} \ar[dl]_(.57){w} \ar@`{+<31pt,12pt>,p+<31pt,-12pt>}[]^{xz} & \\
 \circ \ar@<.4ex>[r]^x \ar@/^46pt/[uuur]^(.58){y} &
 \circ \ar@<.4ex>[r]^x \ar@<.4ex>[l]^z &
 \circ \ar@<.4ex>[l]^z \ar@/_46pt/[uuul]_(.58){w}
}
\\[5pt]
Q(T_0) & Q(T_1)
\end{array}
\]
\end{Ex}

\begin{appendix}

\section{NCCRs of toric singularities with torsion divisor class groups}\label{apptorsion}

For a Gorenstein toric singularity with finite divisor class group,
the direct sum of one divisorial module from each class is the unique
basic reflexive module giving a toric NCCR, up to $R$-module
isomorphism. Its endomorphism algebra has a presentation by a
periodic quiver as in \cite{HIO}.

\begin{Rem}
The existence of an NCCR is well-known when $k$ is an algebraically closed field of characteristic zero since such $R$ is an abelian quotient singularity \cite{Iya07a,LW}. In this case, the quiver of the NCCR is understood as the McKay quiver. Recently, a method to deal with quotient singularities over a non-algebraically closed field has been developed by the author \cite{Tom24}, but there are some restrictions on the characteristic of $k$. Our proof below does not require any conditions on the base field $k$ at all.
\end{Rem}

Let $e_i\in\mathbb{Z}^{d+1}$ be the $i$-th unit vector for $0\leq i\leq d$. Put $\alpha_i:=e_i-e_{i-1}$ for $1\leq i\leq d$ and $\alpha_0:=e_0-e_d$. Let $L:=\{v=(v_i)_{i=0}^d\in\mathbb{Z}^{d+1}\mid\sum_{i=0}^dv_i=0\}=\sum_{i=0}^d\mathbb{Z}\alpha_i\subseteq\mathbb{Z}^{d+1}$ be a $d$-dimensional lattice.

First, we have the following bijections. Compare this with Proposition \ref{3corr}.

\begin{Prop}\label{BGcorr1}
For $d\ge1$, we have a bijection between the following sets.
\begin{enumerate}
\item $\{B\subseteq L\colon\text{cofinite subgroup}\}$
\item $\{(G,(\vec{x}_i)_{i=0}^d)\mid G\text{ is a finite abelian group},\vec{x}_i\in G, G=\sum_{i=0}^d\mathbb{Z}\vec{x}_i,\sum_{i=0}^d\vec{x}_i=0\}/\cong$
\item The $d$-dimensional lattice simplices with ordered vertices, up to affine lattice equivalence preserving that order.
\end{enumerate}
In (2), the group isomorphisms preserve the ordered weights. The bijections for unlabelled simplices are obtained by taking orbits under permutations of the labels in all three sets.
\end{Prop}
\begin{proof}
The $\alpha_i$ generate $L$ with the single relation $\sum_i\alpha_i=0$. Thus (1) and (2) correspond via $G=L/B$ and $\vec{x}_i=\alpha_i+B$, with inverse $B=\Ker(L\to G)$ for $\alpha_i\mapsto\vec{x}_i$.
In (2), the sum-zero relation gives generation after deleting any weight, and finiteness of $G$ gives a positive relation among the remaining weights. Section~\ref{toricGradingData} thus gives (3). All the constructions respect reindexing, which on $L$ permutes the $\alpha_i$.
\end{proof}

In the notation of Proposition \ref{BGcorr1}, let
\[Q:=(G,\bigsqcup_{i=0}^d\{\vec{g}\to\vec{g}+\vec{x}_i\mid\vec{g}\in G\})\]
be a finite quiver which may have multiple arrows. Observe that $Q$ has a Galois covering $\widehat{Q}:=(L,\{x\to x+\alpha_i\mid x\in L, 0\leq i\leq d\})$. Consider the relation $I$ in the path algebra $kQ$ generated by
\[(\vec{g}\to\vec{g}+\vec{x}_i\to\vec{g}+\vec{x}_i+\vec{x}_j)=(\vec{g}\to\vec{g}+\vec{x}_j\to\vec{g}+\vec{x}_i+\vec{x}_j)\]
for $\vec{g}\in G$ and $0\leq i,j\leq d$. Observe that $I\subseteq kQ_{\ge2}$ holds. We put
\[\Gamma(B):=kQ/I.\]

\begin{Thm}\label{abelNCCR}
Let $G$ be a finite abelian group and $\vec{x}_0,\cdots,\vec{x}_d\in G$ elements satisfying $G=\sum_{i=0}^d\mathbb{Z}\vec{x}_i$ and $\sum_{i=0}^d\vec{x}_i=0$. We view the polynomial ring $S=k[x_0,\cdots,x_d]$ as a $G$-graded $k$-algebra by $\deg x_i=\vec{x}_i$.
\begin{enumerate}
\item $R:=S_0$ is a $(d+1)$-dimensional Gorenstein normal domain.
As an $R$-module, $S$ is maximal Cohen--Macaulay and is the unique
basic reflexive module giving a toric NCCR of $R$, up to isomorphism.
A finite direct sum of divisorial $R$-modules gives a toric NCCR if
and only if each $S_{\vec{g}}$ occurs as a direct summand.
\item Let $B\subseteq L$ be the cofinite subgroup obtained by Proposition \ref{BGcorr1}. Then we have
\[\End_R(S)\cong\Gamma(B).\]
\end{enumerate}
\end{Thm}
\begin{proof}
(1) The grading data give a Gorenstein toric ring by Section~\ref{toricGradingData}. By Proposition \ref{lcohGortoric} and Lemma \ref{torscont}, every $S_{\vec{g}}$ is maximal Cohen--Macaulay. Thus $\Gamma:=\End_R(S)$ is maximal Cohen--Macaulay by Proposition \ref{refhom}. The sum $\bigoplus_{\vec{g}\in G}S(\vec{g})$ is a graded projective generator with endomorphism ring $\Gamma$, so $\mod^G\!S\simeq\mod\Gamma$ and $\gl\Gamma<\infty$.

Let $M$ be a finite direct sum of divisorial $R$-modules and let $J$
be the set of classes of its summands. Then $M\in\add S$, so
$M\oplus S$ is modifying. If $M$ gives a toric NCCR, \cite[4.5]{IW}
gives $S\in\add M=\add M_J$. By \eqref{addDegrees}, this implies
$J=G$. Conversely, if $J=G$, then $\add M=\add S$, so $M$ gives a
toric NCCR. A basic module with $J=G$ is isomorphic to
$\bigoplus_{\vec{g}\in G}S_{\vec{g}}=S$.

(2) The homomorphism $\Gamma(B)\to\Gamma$ of $k$-algebra sending $e_{\vec{g}}$ to ${\rm id}_{S_{\vec{g}}}\in\End_R(S_{\vec{g}})$ and $(\vec{g}\to\vec{g}+\vec{x}_i)$ to $x_i\cdot-\in\Hom_R(S_{\vec{g}},S_{\vec{g}+\vec{x}_i})$ is an isomorphism.
\end{proof}

Although this result is well-known (when $k$ is algebraically closed with $\ch k=0$), we give an example of $\widehat{Q}$ so that we can verify that our argument does not use group representations.

\begin{Ex}
Put $G:=(\mathbb{Z}/2\mathbb{Z})^{\oplus3}, \vec{x}=(1,0,0), \vec{y}=(0,1,0), \vec{z}=(0,0,1), \vec{w}=(1,1,1)\in G$. We view $S:=k[x,y,z,w]$ as a $G$-graded $k$-algebra and define $R:=S_0=k[x^2,y^2,z^2,w^2,xyzw]$. Here, the natural homomorphism $L=\mathbb{Z}^3\to G$ sends $(a,b,c)$ to $(a,b,c)$. Thus $B=\{(a,b,c)\in L\mid a,b,c\in2\mathbb{Z}\}$ holds. Then the quiver $\widehat{Q}$ becomes as follows. Here, the black vertices correspond to elements of $B$.

\begin{center}
\begin{tikzpicture}[
  x={(2.5cm,0cm)},
  y={(-1.5cm,1cm)},
  z={(0cm,2.5cm)},
  >={Stealth[length=2.2mm,width=1.8mm]},
  line cap=round,
  line join=round,
  edge/.style={->, line width=0.95pt},
  rededge/.style={->, red!85!black, line width=1.35pt},
  vertex/.style={circle, fill=black, inner sep=1.25pt},
  every node/.style={font=\scriptsize}
]

\def\N{4}

\foreach \i in {0,...,\N}{
  \foreach \j in {0,...,2}{
    \foreach \k in {0,...,\N}{
      \coordinate (\i-\j-\k) at (\i,\j,\k);
    }
  }
}

\foreach \i in {0,...,\N}{
  \foreach \j in {2}{
    \foreach \k in {0,...,\N}{
      \ifnum\i<\N
        \draw[->, thick] ($(\i-\j-\k)+(0.04,0,0)$) -- ($(\i-\j-\k)+(0.96,0,0)$);
      \fi
      \ifnum\k<\N
        \draw[->, thick] ($(\i-\j-\k)+(0,0,0.04)$) -- ($(\i-\j-\k)+(0,0,0.96)$);
      \fi
      \ifnum\i>0
      \ifnum\k>0
        \draw[white, line width=6pt] ($(\i-\j-\k)+(-0.12,-0.12,-0.12)$) -- ($(\i-\j-\k)+(-0.96,-0.96,-0.96)$);
        \draw[->, thick] ($(\i-\j-\k)+(-0.04,-0.04,-0.04)$) -- ($(\i-\j-\k)+(-0.96,-0.96,-0.96)$);
      \fi
      \fi
    }
  }
}

\foreach \i in {0,...,\N}{
  \foreach \j in {1}{
    \foreach \k in {0,...,\N}{
      \ifnum\i<\N
        \draw[white, line width=6pt] ($(\i-\j-\k)+(0.12,0,0)$) -- ($(\i-\j-\k)+(0.96,0,0)$);
        \draw[->, thick] ($(\i-\j-\k)+(0.04,0,0)$) -- ($(\i-\j-\k)+(0.96,0,0)$);
      \fi
      \draw[white, line width=6pt] ($(\i-\j-\k)+(0,0.2,0)$) -- ($(\i-\j-\k)+(0,0.88,0)$);
      \draw[->, thick] ($(\i-\j-\k)+(0,0.04,0)$) -- ($(\i-\j-\k)+(0,0.96,0)$);
      \ifnum\k<\N
        \draw[white, line width=6pt] ($(\i-\j-\k)+(0,0,0.2)$) -- ($(\i-\j-\k)+(0,0,0.9)$);
        \draw[->, thick] ($(\i-\j-\k)+(0,0,0.04)$) -- ($(\i-\j-\k)+(0,0,0.96)$);
      \fi
      \ifnum\i>0
      \ifnum\k>0
        \draw[white, line width=6pt] ($(\i-\j-\k)+(-0.12,-0.12,-0.12)$) -- ($(\i-\j-\k)+(-0.96,-0.96,-0.96)$);
        \draw[->, thick] ($(\i-\j-\k)+(-0.04,-0.04,-0.04)$) -- ($(\i-\j-\k)+(-0.96,-0.96,-0.96)$);
      \fi
      \fi
    }
  }
}

\foreach \i in {0,...,\N}{
  \foreach \j in {0}{
    \foreach \k in {0,...,\N}{
      \ifnum\i<\N
        \draw[white, line width=6pt] ($(\i-\j-\k)+(0.12,0,0)$) -- ($(\i-\j-\k)+(0.96,0,0)$);
        \draw[->, thick] ($(\i-\j-\k)+(0.04,0,0)$) -- ($(\i-\j-\k)+(0.96,0,0)$);
      \fi
      \draw[white, line width=6pt] ($(\i-\j-\k)+(0,0.2,0)$) -- ($(\i-\j-\k)+(0,0.88,0)$);
      \draw[->, thick] ($(\i-\j-\k)+(0,0.04,0)$) -- ($(\i-\j-\k)+(0,0.96,0)$);
      \ifnum\k<\N
        \draw[white, line width=6pt] ($(\i-\j-\k)+(0,0,0.2)$) -- ($(\i-\j-\k)+(0,0,0.9)$);
        \draw[->, thick] ($(\i-\j-\k)+(0,0,0.04)$) -- ($(\i-\j-\k)+(0,0,0.96)$);
      \fi
    }
  }
}

\foreach \i in {0,...,\N}{
  \foreach \j in {0,...,2}{
    \foreach \k in {0,...,\N}{
    \pgfmathtruncatemacro{\labf}{mod((\i+1)*(\j+1)*(\k+1),2)}
      \ifnum\labf=1
        \node at (\i-\j-\k) {$\bullet$};
      \else
        \node at (\i-\j-\k) {$\circ$};
      \fi
    }
  }
}
\end{tikzpicture}
\end{center}

Here, each arrow of each direction represents the following variables.

\begin{center}
\begin{tikzpicture}[
  x={(2.5cm,0cm)},
  y={(-1.5cm,1cm)},
  z={(0cm,2.5cm)},
  >={Stealth[length=2.2mm,width=1.8mm]},
  line cap=round,
  line join=round,
  edge/.style={->, line width=0.95pt},
  rededge/.style={->, red!85!black, line width=1.35pt},
  vertex/.style={circle, fill=black, inner sep=1.25pt},
  every node/.style={font=\scriptsize}
]

\node at (0,0,0) {$\circ$};

\draw[->, thick] (0.04,0,0) -- node[midway,above] {$x$} (0.96,0,0);
\draw[->, thick] (0,0.04,0) -- node[midway,above] {$y$} (0,0.96,0);
\draw[->, thick] (0,0,0.04) -- node[midway,right] {$z$} (0,0,0.96);
\draw[->, thick] (-0.04,-0.04,-0.04) -- node[midway,right] {$w$} (-0.96,-0.96,-0.96);
\end{tikzpicture}
\end{center}

Therefore the quiver $Q$ of the toric NCCR is as follows, where each segment represents a double arrow.

\begin{center}
\begin{tikzpicture}[
  x={(2.5cm,0cm)},
  y={(-1.5cm,1cm)},
  z={(0cm,2.5cm)},
  >={Stealth[length=2.2mm,width=1.8mm]},
  line cap=round,
  line join=round,
  edge/.style={->, line width=0.95pt},
  rededge/.style={->, red!85!black, line width=1.35pt},
  vertex/.style={circle, fill=black, inner sep=1.25pt},
  every node/.style={font=\scriptsize}
]

\foreach \i in {0,1}{
  \foreach \j in {0,1}{
    \foreach \k in {0,1}{
      \coordinate (\i-\j-\k) at (\i,\j,\k);
    }
  }
}

\draw[-, thick] (0.04,1,0) -- node[midway,below] {$x$} (0.96,1,0);
\draw[-, thick] (0.04,1,1) -- node[midway,above] {$x$} (0.96,1,1);
\draw[-, thick] (0,1,0.04) -- node[midway,left] {$z$} (0,1,0.96);
\draw[-, thick] (1,1,0.04) -- node[midway,right] {$z$} (1,1,0.96);

\draw[-, thick] (0,0.04,0) -- node[midway,left] {$y$} (0,0.96,0);
\draw[-, thick] (1,0.04,0) -- node[midway,left] {$y$} (1,0.96,0);
\draw[-, thick] (0,0.04,1) -- node[midway,right] {$y$} (0,0.96,1);
\draw[-, thick] (1,0.04,1) -- node[midway,right] {$y$} (1,0.96,1);

\draw[white, line width=5pt] (0.1,0.1,0.1) -- (0.9,0.9,0.9);
\draw[white, line width=5pt] (0.8,0.2,0.2) -- (0.2,0.8,0.8);
\draw[white, line width=5pt] (0.96,0.04,0.96) -- (0.1,0.9,0.1);
\draw[-, thick] (0.04,0.04,0.04) -- node[midway,right] {$w$} (0.96,0.96,0.96);
\draw[-, thick] (0.96,0.04,0.04) -- (0.04,0.96,0.96);
\draw[-, thick] (0.04,0.04,0.96) -- (0.96,0.96,0.04);
\draw[-, thick] (0.96,0.04,0.96) -- (0.04,0.96,0.04);

\draw[-, thick] (0.04,0,0) -- node[midway,below] {$x$} (0.96,0,0);
\draw[white, line width=5pt] (0.1,0,1) -- (0.88,0,1);
\draw[-, thick] (0.04,0,1) -- node[midway,above] {$x$} (0.96,0,1);
\draw[white, line width=5pt] (0,0,0.15) -- (0,0,0.9);
\draw[-, thick] (0,0,0.04) -- node[midway,left] {$z$} (0,0,0.96);
\draw[-, thick] (1,0,0.04) -- node[midway,right] {$z$} (1,0,0.96);

\foreach \i in {0,1}{
  \foreach \j in {0,1}{
    \foreach \k in {0,1}{
      \node at (\i-\j-\k) {$\circ$};
    }
  }
}

\end{tikzpicture}
\end{center}

\end{Ex}

\section{Volumes of $d$-dimensional lattice polytopes with $d+2$ vertices}\label{appvolume}

In this appendix, we give a formula for the volume of $d$-dimensional lattice polytopes with $d+2$ vertices (Theorem \ref{vol}).

Let $N\cong\mathbb{Z}^d$ be a $d$-dimensional lattice and $P\subseteq N_{\mathbb{R}}$ a $d$-dimensional lattice polytope with vertices $\{v_1,\cdots,v_{d+2}\}$. Put $\widetilde{N}:=N\oplus\mathbb{Z}$ and $\widetilde{v}_j:=(v_j,1)\in\widetilde{N}$. Let $\phi\colon\mathbb{Z}^{d+2}=\bigoplus_{j=1}^{d+2}\mathbb{Z}e_j\to\widetilde{N}$ be the group homomorphism sending $e_i$ to $\widetilde{v}_i$. Define $G:=\Cok\phi^*$.
\[0\to\widetilde{N}^*\xrightarrow{\phi^*}(\mathbb{Z}^{d+2})^*\to G\to0\]
Let $\vec{z}_j\in G$ be the image of the $j$-th unit vector of $(\mathbb{Z}^{d+2})^*$. Let $\pi\colon G\to G/G_{\rm tors}\cong\mathbb{Z}$. Define integers $l,l'\ge1$ so that $(G,(\vec{z}_j)_j)$ satisfies the conditions (G1), (G2) and (G3), where
\[\vec{z}_j=\left\{
\begin{array}{ll}
\vec{x}_{j-1} & (1\leq j\leq l+1)\\
\vec{x}'_{j-l-2} & (l+2\leq j\leq l+l'+2)\\
\vec{y}_{j-l-l'-2} & (l+l'+3\leq j\leq d+2)
\end{array}
\right..\]
Put $H:=G/(\sum_{i''=1}^{d-l-l'}\mathbb{Z}\vec{y}_{i''})$ and let $q\colon G\to H$ be the natural surjection. Let $s:=\sum_{i=0}^lq(\vec{x}_i)=-\sum_{i'=0}^{l'}q(\vec{x}'_{i'})\in H$.

Our main theorem in this appendix is the following.

\begin{Thm}\label{vol}
If we put $K:=\sum_{i''=1}^{d-l-l'}\mathbb{Z}\vec{y}_{i''}$, then we have
\[\vol(P)=\frac{1}{d!}|K||H/\mathbb{Z}s|.\]
\end{Thm}

Removing a pyramid apex reduces the dimension and quotients out the corresponding torsion weight.

\begin{Lem}\label{pyramidVolumeReduction}
Assume $d>l+l'$ and put $P':=\Conv\{v_1,\ldots,v_{d+1}\}$. Then $P'$ is the base of the pyramid with apex $v_{d+2}$, and its grading group is $G/\mathbb{Z}\vec{z}_{d+2}$. With volume on $P'$ measured in its affine lattice, we have
\[
d!\vol(P)=\ord_G(\vec{z}_{d+2})\,(d-1)!\vol(P').
\]
\end{Lem}
\begin{proof}
Let $L$ be the annihilator of $\widetilde v_1,\ldots,\widetilde v_{d+1}$ in $\widetilde N^*$. By the definition of $G$,
\[
\{f(\widetilde v_{d+2})\mid f\in L\}=\ord_G(\vec{z}_{d+2})\mathbb{Z}.
\]
Since $\vec{z}_{d+2}$ is torsion, $L$ has rank one and $P'$ is a facet. A primitive generator of $L$ evaluates on the apex, up to sign, to its lattice height over $P'$, giving the volume identity.

The homogenized lattice of $P'$ is the saturated sublattice $\widetilde N\cap\operatorname{span}_{\mathbb{R}}\{\widetilde v_1,\ldots,\widetilde v_{d+1}\}$. Restriction of functionals to it is surjective. Deleting the last coordinate in the presentation of $G$ therefore gives the grading group of $P'$ as $G/\mathbb{Z}\vec{z}_{d+2}$.
\end{proof}

Iterating Lemma \ref{pyramidVolumeReduction} removes all torsion weights. Their successive orders are the orders of the cyclic factors in the filtration of $K$ by the subgroups they generate, so their product is $|K|$. The resulting group is $H$, and the sum of the positive weights maps to $s$. It therefore suffices to prove the formula for $d=l+l'$, when $G=H$. Put
\[w_i:=v_{i+1}\ (0\le i\le l)\text{ and }w'_{i'}:=v_{i'+l+2}\ (0\le i'\le l').\]

\begin{Lem}\label{divide}
For $0\leq i\leq l$, let $P_i\subseteq N_{\mathbb{R}}$ be the lattice simplex with vertices $\{w_0,\cdots,\widehat{w}_i,\cdots,w_l,w'_0,\cdots,w'_{l'}\}$.
\begin{enumerate}
\item We have
\[P=\bigcup_{i=0}^lP_i.\]
\item For $0\leq i_1<i_2\leq l$, we have
\[P_{i_1}\cap P_{i_2}=\Conv\{w_0,\cdots,\widehat{w}_{i_1},\cdots,\widehat{w}_{i_2},\cdots,w_l,w'_0,\cdots,w'_{l'}\}.\]
\end{enumerate}
\end{Lem}
\begin{proof}
(1) Note that there exist positive rational numbers $\lambda_i,\lambda'_{i'}$ such that
\[
\sum_{i=0}^l \lambda_i w_i
=
\sum_{i'=0}^{l'} \lambda'_{i'} w'_{i'}
\quad\text{and}\quad
\sum_{i=0}^l \lambda_i
=
\sum_{i'=0}^{l'} \lambda'_{i'}.
\]
(See \cite[5.4,5.5]{BH}.) Take $v=\sum_{i=0}^la_iw_i+\sum_{i'=0}^{l'}a'_{i'}w'_{i'}\in P$ where $a_i\ge0,a'_{i'}\ge0$ and $\sum_{i=0}^la_i+\sum_{i'=0}^{l'}a'_{i'}=1$.
Put
\[
t=\min_{0\le i\le l} \frac{a_i}{\lambda_i}
\]
and choose $i_0$ attaining this minimum. Put
\[
b_i=a_i-t\lambda_i,\qquad b'_{i'}=a'_{i'}+t\lambda'_{i'}.
\]
Then all $b_i,b'_{i'}$ are non-negative, $b_{i_0}=0$, and the above affine relation shows that
\[
v=\sum_{i=0}^l b_iw_i+\sum_{i'=0}^{l'}b'_{i'}w'_{i'}.
\]
Thus $v\in P_{i_0}$.

(2) Take $v=\sum_{i=0}^la_iw_i+\sum_{i'=0}^{l'}a'_{i'}w'_{i'}=\sum_{i=0}^lb_iw_i+\sum_{i'=0}^{l'}b'_{i'}w'_{i'}\in P_{i_1}\cap P_{i_2}$ where $a_i\ge0, a'_{i'}\ge0,b_i\ge0,b'_{i'}\ge0,\sum_{i=0}^la_i+\sum_{i'=0}^{l'}a'_{i'}=\sum_{i=0}^lb_i+\sum_{i'=0}^{l'}b'_{i'}=1$ and $a_{i_1}=b_{i_2}=0$. Observe that there exist integers $m_1,m_2>0$ such that $m_1\vec{x}_{i_1}=m_2\vec{x}_{i_2}\in G$ holds. This means that there exists $f\in\widetilde{N}^*$ such that $f(\widetilde{w}_i)=\left\{
\begin{array}{ll}
m_1 & (i=i_1)\\
-m_2 & (i=i_2)\\
0 & (i\neq i_1,i_2)
\end{array}
\right.$ and $f(\widetilde{w}'_{i'})=0$. Note that $\sum_{i=0}^la_i\widetilde{w}_i+\sum_{i'=0}^{l'}a'_{i'}\widetilde{w}'_{i'}=(v,1)=\sum_{i=0}^lb_i\widetilde{w}_i+\sum_{i'=0}^{l'}b'_{i'}\widetilde{w}'_{i'}$ holds. Thus we have
\[-m_2a_{i_2}=f(v,1)=m_1b_{i_1}.\]
This forces $a_{i_2}=b_{i_1}=0$.
\end{proof}

\begin{proof}[Proof of Theorem \ref{vol}]
By Lemma \ref{divide}, $\vol(P)=\sum_{i=0}^l\vol(P_i)$ holds. Thus we have
\[d!\vol(P)=\sum_{i=0}^l|\det[\widetilde{w}_0,\cdots,\widehat{\widetilde{w}}_i,\cdots,\widetilde{w}_l,\widetilde{w}'_0,\cdots,\widetilde{w}'_{l'}]|.\]
Put $d_j:=(-1)^j\det[\widetilde{v}_1,\cdots,\widehat{\widetilde{v}}_j,\cdots,\widetilde{v}_{d+2}]$ for $1\leq j\leq d+2$. Then we have $\sum_{j=1}^{d+2}d_j\widetilde{v}_j=0$. For $0\leq i_1<i_2\leq l$, take $f\in\widetilde{N}^*$ as in the proof of Lemma \ref{divide}(2). Then we have
\[0=f(\sum_{j=1}^{d+2}d_j\widetilde{v}_j)=d_{i_1+1}m_1-d_{i_2+1}m_2.\]
This means that $d_{i_1+1}$ and $d_{i_2+1}$ have the same signs. Therefore we obtain
\begin{align*}
d!\vol(P)&=\sum_{j=1}^{l+1}|d_j| \\
&=|\sum_{j=1}^{l+1}d_j| \\
&=|\det[\widetilde{v}_2-\widetilde{v}_1,\cdots,\widetilde{v}_{l+1}-\widetilde{v}_l,\widetilde{v}_{l+2},\cdots,\widetilde{v}_{d+2}]|.
\end{align*}
Let $\iota\colon\mathbb{Z}^{d+1}=\bigoplus_{j=1}^{d+1}\mathbb{Z}e'_j\to\mathbb{Z}^{d+2}$ be the group homomorphism defined by
\[\iota(e'_j)=\left\{
\begin{array}{ll}
e_{j+1}-e_j & (1\leq j\leq l)\\
e_{j+1} & (l+1\leq j\leq d+1)
\end{array}
\right..\]
Put $\psi:=\phi\iota\colon\mathbb{Z}^{d+1}\to\widetilde{N}$. Then we can check that $\Cok\psi^*\cong G/\mathbb{Z}s$ holds by the following diagram.
\[\xymatrix{
 & & 0 \ar[d] & 0 \ar[d] \\
 & & \mathbb{Z} \ar@{=}[r] \ar[d] & \mathbb{Z} \ar[d] \\
0 \ar[r] & (\widetilde{N})^* \ar[r]^{\phi^*} \ar@{=}[d] & (\mathbb{Z}^{d+2})^* \ar[r] \ar[d]^{\iota^*} & G \ar[r] \ar[d] & 0 \\
0 \ar[r] & (\widetilde{N})^* \ar[r]^{\psi^*} & (\mathbb{Z}^{d+1})^* \ar[r] \ar[d] & G/\mathbb{Z}s \ar[r] \ar[d] & 0 \\
 & & 0 & 0
}\]
Thus we have
\begin{align*}
d!\vol(P)&=|\det[\widetilde{v}_2-\widetilde{v}_1,\cdots,\widetilde{v}_{l+1}-\widetilde{v}_l,\widetilde{v}_{l+2},\cdots,\widetilde{v}_{d+2}]| \\
&=|\det\psi| \\
&=|\det\psi^*| \\
&=|G/\mathbb{Z}s|.\qedhere
\end{align*}
\end{proof}

\end{appendix}

\bibliographystyle{alpha} 
\bibliography{reference}

\newcommand{\etalchar}[1]{$^{#1}$}
\begin{thebibliography}{vVdBB22}

\bibitem[AIR15]{AIR15}
Claire Amiot, Osamu Iyama, and Idun Reiten.
\newblock Stable categories of {C}ohen-{M}acaulay modules and cluster
  categories: {D}edicated to {R}agnar-{O}laf {B}uchweitz on the occasion of his
  sixtieth birthday.
\newblock {\em American Journal of Mathematics}, 137(3):813--857, 2015.

\bibitem[BH98]{BrH98}
Winfried Bruns and J{\"u}rgen Herzog.
\newblock {\em Cohen--Macaulay rings}, volume~39 of {\em Cambridge Studies in
  Advanced Mathematics}.
\newblock Cambridge University Press, revised edition, 1998.

\bibitem[BH09]{BH}
Lev Borisov and Zheng Hua.
\newblock On the conjecture of {K}ing for smooth toric {D}eligne--{M}umford
  stacks.
\newblock {\em Advances in Mathematics}, 221(1):277--301, 2009.

\bibitem[BKR01]{BKR}
Tom Bridgeland, Alastair King, and Miles Reid.
\newblock The {M}c{K}ay correspondence as an equivalence of derived categories.
\newblock {\em Journal of the American Mathematical Society}, 14(3):535--554,
  2001.

\bibitem[BMR{\etalchar{+}}06]{BMRRT}
Aslak~Bakke Buan, Bethany Marsh, Markus Reineke, Idun Reiten, and Gordana
  Todorov.
\newblock Tilting theory and cluster combinatorics.
\newblock {\em Advances in mathematics}, 204(2):572--618, 2006.

\bibitem[Boc12a]{BocCons12}
Raf Bocklandt.
\newblock Consistency conditions for dimer models.
\newblock {\em Glasgow Mathematical Journal}, 54:429--447, 2012.

\bibitem[Boc12b]{Boc12}
Raf Bocklandt.
\newblock Generating toric noncommutative crepant resolutions.
\newblock {\em Journal of Algebra}, 364:119--147, 2012.

\bibitem[Bro12]{Bro}
Nathan Broomhead.
\newblock Dimer models and {C}alabi--{Y}au algebras.
\newblock {\em Memoirs of the American Mathematical Society},
  215(1011):viii+86, 2012.

\bibitem[CQV12]{CQV12}
Alastair Craw and Alexander Quintero~V\'{e}lez.
\newblock Cellular resolutions of noncommutative toric algebras from
  superpotentials.
\newblock {\em Advances in Mathematics}, 229(3):1516--1554, 2012.

\bibitem[DG26]{DG}
Darius Dramburg and Oleksandra Gasanova.
\newblock A classification of $n$-representation infinite algebras of type
  $\widetilde{A}$.
\newblock {\em Mathematische Zeitschrift}, 313:62, 2026.

\bibitem[Han24]{Han24a}
Norihiro Hanihara.
\newblock Higher hereditary algebras and {C}alabi--{Y}au algebras arising from
  some toric singularities.
\newblock arXiv:2412.19040, 2024.

\bibitem[Han25]{Han25}
Norihiro Hanihara.
\newblock Non-commutative resolutions for {S}egre products and
  {C}ohen-{M}acaulay rings of hereditary representation type.
\newblock {\em Transactions of the American Mathematical Society},
  378(4):2429--2475, 2025.

\bibitem[Har17]{Har}
Wahei Hara.
\newblock Non-commutative crepant resolution of minimal nilpotent orbit
  closures of type {A} and {M}ukai flops.
\newblock {\em Advances in mathematics}, 318:355--410, 2017.

\bibitem[HH24]{HH24}
Wahei Hara and Yuki Hirano.
\newblock Mutations of noncommutative crepant resolutions in geometric
  invariant theory.
\newblock {\em Selecta Mathematica}, 30(4):70, 2024.

\bibitem[HI22]{HI22}
Norihiro Hanihara and Osamu Iyama.
\newblock Enhanced {A}uslander--{R}eiten duality and {M}orita theorem for
  singularity categories.
\newblock arXiv:2209.14090, 2022.

\bibitem[HIO14]{HIO}
Martin Herschend, Osamu Iyama, and Steffen Oppermann.
\newblock {$n$}-representation infinite algebras.
\newblock {\em Advances in Mathematics}, 252:292--342, 2014.

\bibitem[HN19]{HN}
Akihiro Higashitani and Yusuke Nakajima.
\newblock Conic divisorial ideals of {H}ibi rings and their applications to
  non-commutative crepant resolutions.
\newblock {\em Selecta Mathematica}, 25(5):78, 2019.

\bibitem[IR08]{IR}
Osamu Iyama and Idun Reiten.
\newblock Fomin--{Z}elevinsky mutation and tilting modules over {C}alabi--{Y}au
  algebras.
\newblock {\em American Journal of Mathematics}, 130(4):1087--1149, 2008.

\bibitem[IW14]{IW}
Osamu Iyama and Michael Wemyss.
\newblock Maximal modifications and {A}uslander-{R}eiten duality for
  non-isolated singularities.
\newblock {\em Inventiones mathematicae}, 197(3):521--586, 2014.

\bibitem[Iya07a]{Iya07b}
Osamu Iyama.
\newblock Auslander correspondence.
\newblock {\em Advances in mathematics}, 210(1):51--82, 2007.

\bibitem[Iya07b]{Iya07a}
Osamu Iyama.
\newblock Higher-dimensional {A}uslander-{R}eiten theory on maximal orthogonal
  subcategories.
\newblock {\em Advances in mathematics}, 210(1):22--50, 2007.

\bibitem[LW12]{LW}
Graham~J. Leuschke and Roger Wiegand.
\newblock {\em Cohen--Macaulay representations}, volume 181 of {\em
  Mathematical Surveys and Monographs}.
\newblock American Mathematical Society, Providence, RI, 2012.

\bibitem[Mat22]{Mat22}
Koji Matsushita.
\newblock Conic divisorial ideals of toric rings and applications to {H}ibi
  rings and stable set rings.
\newblock arXiv:2210.02031, 2022.

\bibitem[MS25]{MS25}
Aimeric Malter and Artan Sheshmani.
\newblock Towards non-commutative crepant resolutions of affine toric
  {G}orenstein varieties.
\newblock arXiv:2509.11664, 2025.

\bibitem[MS26]{MS26}
Aimeric Malter and Artan Sheshmani.
\newblock Non-commutative crepant resolutions for (almost) simplicial toric
  algebras.
\newblock arXiv:2602.21802, 2026.

\bibitem[Nak19]{Nak19}
Yusuke Nakajima.
\newblock Mutations of splitting maximal modifying modules: The case of
  reflexive polygons.
\newblock {\em International Mathematics Research Notices}, 2019(2):470--550,
  2019.

\bibitem[NVdB26]{NVdB26}
Anya Nordskova and Michel Van~den Bergh.
\newblock N{CCR}s of cones over del {P}ezzo surfaces.
\newblock arXiv:2604.11319, 2026.

\bibitem[Sta83]{Sta}
Richard~P. Stanley.
\newblock {\em Combinatorics and Commutative Algebra}, volume~41 of {\em
  Progress in Mathematics}.
\newblock Birkh{\"a}user, 1983.

\bibitem[Tom24]{Tom24}
Ryu Tomonaga.
\newblock Cohen-{M}acaulay representations of invariant subrings.
\newblock arXiv:2403.19282, 2024.

\bibitem[Tom25]{Tom25a}
Ryu Tomonaga.
\newblock Weak del {P}ezzo surfaces are characterized by the existence of
  $2$-tilting bundles.
\newblock arXiv:2510.26199, 2025.

\bibitem[Tom26]{Tom25b}
Ryu Tomonaga.
\newblock Higher representation infinite algebras and toric {F}ano stacks of
  {P}icard number one or two.
\newblock arXiv:2511.02641v5, 2026.

\bibitem[VdB04a]{VdB04a}
Michel Van~den Bergh.
\newblock Non-commutative crepant resolutions.
\newblock In Olav~Arnfinn Laudal and Ragni Piene, editors, {\em The Legacy of
  {Niels Henrik Abel}: The Abel Bicentennial, Oslo, 2002}, pages 749--770.
  Springer, Berlin, 2004.

\bibitem[VdB04b]{VdB04b}
Michel Van~den Bergh.
\newblock Three-dimensional flops and noncommutative rings.
\newblock {\em Duke Math. J.}, 122(3):423--455, 2004.

\bibitem[VdB23]{VdB23}
Michel Van~den Bergh.
\newblock Noncommutative crepant resolutions, an overview.
\newblock In Dmitry Beliaev and Stanislav Smirnov, editors, {\em Proceedings of
  the International Congress of Mathematicians 2022}, volume~2, pages
  1354--1391, Berlin, 2023. EMS Press.

\bibitem[vVdB17]{SVdB17}
\v{S}pela \v{S}penko and Michel Van~den Bergh.
\newblock Non-commutative resolutions of quotient singularities for reductive
  groups.
\newblock {\em Inventiones mathematicae}, 210:3--67, 2017.

\bibitem[vVdB20a]{SVdB20a}
\v{S}pela \v{S}penko and Michel Van~den Bergh.
\newblock Non-commutative crepant resolutions for some toric singularities {I}.
\newblock {\em International Mathematics Research Notices},
  2020(21):8120--8138, 2020.

\bibitem[vVdB20b]{SVdB20b}
\v{S}pela \v{S}penko and Michel Van~den Bergh.
\newblock Non-commutative crepant resolutions for some toric singularities
  {II}.
\newblock {\em Journal of Noncommutative Geometry}, 14(1):73--103, 2020.

\bibitem[vVdBB22]{SVdBB22}
\v{S}pela \v{S}penko, Michel Van~den Bergh, and Jason~P. Bell.
\newblock On the noncommutative {Bondal--Orlov} conjecture for some toric
  varieties.
\newblock {\em Mathematische Zeitschrift}, 300(1):1055--1068, 2022.

\bibitem[Yos90]{Yos90}
Yuji Yoshino.
\newblock {\em Cohen--Macaulay modules over Cohen--Macaulay rings}, volume 146
  of {\em London Mathematical Society Lecture Note Series}.
\newblock Cambridge University Press, 1990.

\end{thebibliography}

\end{document}